\documentclass[a4paper,11pt,reqno]{amsart}

\usepackage[margin=1in,includehead,includefoot]{geometry}
\usepackage[utf8]{inputenc}
\usepackage[T1]{fontenc}
\usepackage{lmodern,microtype}
\usepackage{hyperref}
\usepackage{url}
\usepackage{booktabs}
\usepackage{mathtools,amssymb,amsthm,amsmath}
\usepackage[foot]{amsaddr}
\usepackage{graphicx, color}
\usepackage[margin=5mm]{caption}
\usepackage{enumitem}
\usepackage{todonotes}
\usepackage[algosection,algoruled,vlined,nofillcomment]{algorithm2e}
\usepackage{tikz,pgfplots}
\usepackage{mycommands}

\graphicspath{{./graphics/}}

\usetikzlibrary{shapes.geometric,shapes.multipart,shapes.misc,positioning,decorations.pathreplacing,plotmarks,patterns,calc,matrix}
\pgfplotsset{compat=1.11}

\DontPrintSemicolon
\SetKwProg{myfunc}{Function}{}{}

\newtheorem{theorem}{Theorem}
\newtheorem{lemma}[theorem]{Lemma}
\newtheorem{conjecture}[theorem]{Conjecture}

\hypersetup{
  pdftitle={Matchings in hypercubes extend to long cycles},
  pdfauthor={Ji\v{r}\'i Fink, Torsten M\"utze}
}

\title{Matchings in hypercubes extend to long cycles}

\author{Ji\v{r}\'i Fink}
\address[Ji\v{r}\'i Fink]{Department of Theoretical Computer Science and Mathematical Logic, Charles University, Prague, Czech Republic}
\email{fink@ktiml.mff.cuni.cz}

\author{Torsten M\"utze}
\address[Torsten M\"utze]{Department of Computer Science, University of Warwick, United Kingdom \& Department of Theoretical Computer Science and Mathematical Logic, Charles University, Prague, Czech Republic}
\email{torsten.mutze@warwick.ac.uk}

\thanks{An extended abstract of this paper appeared in the Proceedings of IWOCA~2024~\cite{MR4786527}.}
\thanks{This work was supported by Czech Science Foundation grant GA~22-15272S. Both authors participated in the workshop `Combinatorics, Algorithms and Geometry' in March 2024, which was funded by German Science Foundation grant~522790373.}

\begin{document}

\begin{abstract}
The $d$-dimensional hypercube graph~$Q_d$ has as vertices all subsets of~$\{1,\ldots,d\}$, and an edge between any two sets that differ in a single element.
The Ruskey-Savage conjecture asserts that every matching of~$Q_d$, $d\ge 2$, can be extended to a Hamilton cycle, i.e., to a cycle that visits every vertex exactly once.
We prove that every matching of~$Q_d$, $d\ge 2$, can be extended to a cycle that visits at least a~$2/3$-fraction of all vertices.
\end{abstract}

\maketitle

\section{Introduction}

Cycles and matchings in graphs are structures of fundamental interest.
In this paper, we consider these structures in hypercubes, a family of graphs that has been studied widely in computer science and mathematics.
Specifically, the \defi{$d$-dimensional hypercube~$Q_d$} is the graph whose vertices are all subsets of~$[d]:=\{1,\ldots,d\}$ and whose edges connect sets that differ in a single element; see Figure~\ref{fig:rs}~(a).
It is well-known and easy to show that~$Q_d$, $d\ge 2$, admits a \defi{Hamilton cycle}, i.e., a cycle that visits every vertex exactly once.
Clearly, any Hamilton cycle in~$Q_d$ is the union of two perfect matchings.
A \defi{matching} in a graph is a set of edges that are pairwise disjoint, and a matching is \defi{perfect} if it includes every vertex of the graph.

\begin{figure}
\centerline{
\includegraphics[page=1]{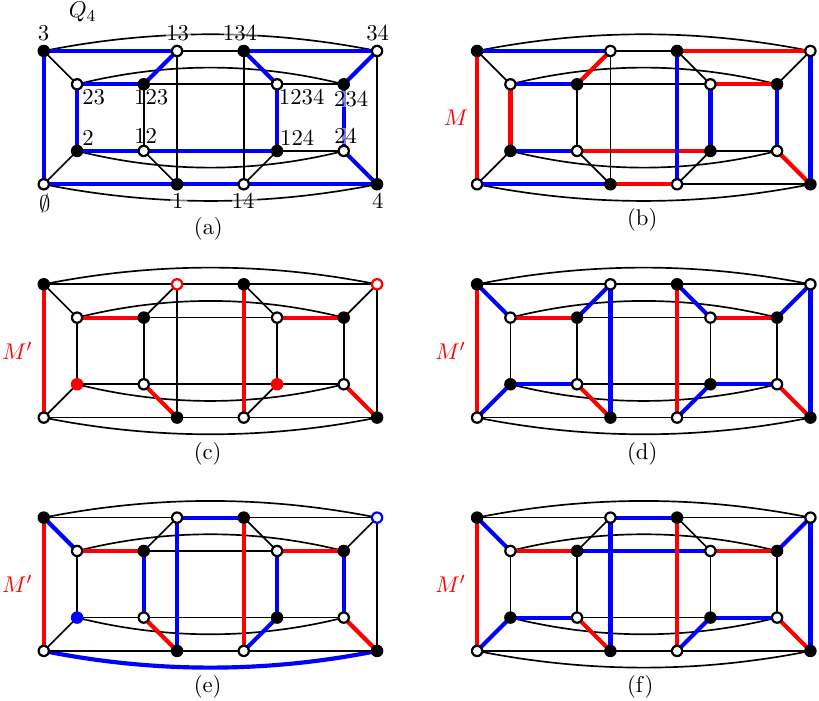}
}
\caption{Cycles and matchings in the hypercube~$Q_4$: (a)~a Hamilton cycle, where the vertex labels omit curly brackets and commas for conciseness;
(b)~a perfect matching~$M$ and an extension of~$M$ to a Hamilton cycle;
(c)~a maximal, but not perfect matching~$M'$;
(d)~an extension of~$M'$ to a cycle factor with two cycles;
(e)~an extension of~$M'$ to a (non-Hamilton) cycle;
(f)~an extension of~$M'$ to a Hamilton cycle, obtained by joining the two cycles in~(d).}
\label{fig:rs}
\end{figure}

30 years ago, Ruskey and Savage~\cite{MR1201997} asked whether every matching of~$Q_d$ can be extended to a Hamilton cycle.
This became known as the \defi{Ruskey-Savage conjecture}.

\begin{conjecture}[\cite{MR1201997}]
\label{conj:ruskey-savage}
Every matching of~$Q_d$, $d\ge 2$, can be extended to a Hamilton cycle.
\end{conjecture}

This problem received considerable attention, and several natural relaxations have been proved.
In particular, Fink~\cite{MR2354719} settled the conjecture affirmatively for the case when the prescribed matching is perfect, thereby answering a problem due to Kreweras~\cite{MR1374632}; see Figure~\ref{fig:rs}~(b).
In fact, Fink established a considerable strengthening, obtained by considering the graph~$K(Q_d)$, which is the complete graph on the vertex set of~$Q_d$.
In this context, we say that a matching~$M$ of~$K(Q_d)$ \defi{extends to} a Hamilton cycle~$C$ (or some other structure) if all edges in~$C\setminus M$ belong to~$Q_d$.
Put differently, while the matching~$M$ may use arbitrary edges of~$K(Q_d)$, including edges not present in~$Q_d$, the edges $C\setminus M$ used for the extension are required to be edges of the hypercube~$Q_d$ (otherwise $C$ would trivially be present in the complete graph).

\begin{theorem}[\cite{MR2354719}]
\label{thm:kreweras}
Every perfect matching of~$K(Q_d)$, $d\ge 2$, can be extended to a Hamilton cycle.
\end{theorem}

In fact, the strengthening to consider~$K(Q_d)$ instead of~$Q_d$ is the key idea to Fink's proof, as it makes the induction hypothesis stronger and thus more flexible.
This construction actually shows that every perfect matching of~$K(Q_d)$ can be extended to at least $2^{2^{d-4}}$ distinct Hamilton cycles.
This proof technique has been exploited in several subsequent papers (see e.g.~\cite{MR2510579,MR3830138,MR3337323}), and we will also use it heavily in our arguments.

Clearly, not every matching in~$Q_d$ is perfect or extends to a perfect matching.
In other words, there are (inclusion-)maximal matchings in~$Q_d$ that are not perfect; see Figure~\ref{fig:rs}~(c).
Therefore, Theorem~\ref{thm:kreweras} leaves open the question whether every (not necessarily perfect) matching extends to a Hamilton cycle; see Figure~\ref{fig:rs}~(f).
Dvo\v{r}\'{a}k and Fink~\cite{MR3891930} obtained some positive evidence for small matchings.

\begin{theorem}[\cite{MR3891930}]
Every matching of~$Q_d$, $d\ge 2$, with at most~$d^2/16+d/4$ edges can be extended to a Hamilton cycle.
\end{theorem}

Another relaxation of the Ruskey-Savage conjecture, proposed by Vandenbussche and West~\cite{MR3089715}, is to consider extensions to a \defi{cycle factor}, i.e., a collection of disjoint cycles that together visit all vertices of the graph; see Figure~\ref{fig:rs}~(d).
This variant of the problem was also settled by Fink~\cite{MR3936192}.

\begin{theorem}[\cite{MR3936192}]
\label{thm:vbwest}
Every matching of~$Q_d$, $d\ge 2$, can be extended to a cycle factor.
\end{theorem}

\subsection{Our results}
\label{sec:results}

In this paper we prove that every matching of~$Q_d$ extends to a single cycle~$C$.
However, $C$ possibly omits some vertices of~$Q_d$, i.e., $C$ is not necessarily a Hamilton cycle; see Figure~\ref{fig:rs}~(e).
Our theorem also holds in the stronger setting of the complete graph~$K(Q_d)$.

\begin{theorem}
\label{thm:cycle}
Every matching of~$K(Q_d)$, $d\ge 2$, can be extended to a cycle.
\end{theorem}

Furthermore, we can give some guarantee about the length of the cycle obtained from our construction, namely that it visits a constant fraction of the vertex set~$V(Q_d)$.
This follows by using that any matching can be extended to a matching that contains a constant fraction of all vertices.

\begin{theorem}
\label{thm:long-Qd}
Every matching of~$Q_d$, $d\ge 2$, can be extended to a cycle of length at least~$\frac{2}{3}|V(Q_d)|=2^{d+1}/3$.
\end{theorem}

\begin{theorem}
\label{thm:long-KQd}
Every matching of~$K(Q_d)$, $d\ge 2$, can be extended to a cycle of length at least~$\frac{1}{2}|V(Q_d)|=2^{d-1}$.
\end{theorem}

Our results are another step towards the Ruskey-Savage conjecture, and we want to point out an analogy to the \defi{middle levels conjecture}, which asserts that the subgraph of~$Q_{2d+1}$ induced by all sets of size~$d$ or~$d+1$ admits a Hamilton cycle.
Historically, Felsner and Trotter~\cite{MR1350586}, Savage and Winkler~\cite{MR1329390}, and Johnson~\cite{MR2046083} first established the existence of long cycles in this graph, and their work subsequently led to a proof of the middle levels conjecture~\cite{MR3483129}.

We also note that not every matching of~$K(Q_d)$ can be extended to a Hamilton cycle, nor a cycle factor.
In fact, any matching in~$K(Q_d)$ between all vertices of the same parity has the property that any cycle extending it includes only half of the vertices from the other parity, so it has only length~$\frac{3}{4}|V(Q_d)|$.
The \defi{parity} of a vertex~$u$ of~$Q_n$ is the parity of~$|u|$, i.e., of the size of the set~$u$.
In particular, in~$K(Q_2)$, the matching with the single edge~$(\emptyset,\{1,2\})$ can only be extended to a cycle of length~3.

Clearly, if a matching~$M$ of~$K(Q_d)$ is extendable to a Hamilton cycle, then $M$ has the same number of vertices of each parity.
Note that every matching of~$B(Q_d)$ satisfies this condition, where~$B(Q_d)$ is the complete bipartite graph obtained from~$Q_d$ by adding all edges between vertices of opposite parity.
However, Dvo\v{r}\'{a}k and Fink~\cite{MR3891930} showed that this condition is not sufficient, by constructing a matching of~$B(Q_d)$ for $d \ge 9$ which cannot be extended to a Hamilton cycle, nor a cycle factor.

To prove Theorem~\ref{thm:cycle} inductively, we use the following auxiliary theorem, which yields a cycle that extends a given matching but also avoids one `forbidden' vertex~$z$, i.e., avoiding~$z$ is an additional constraint imposed on the cycle.
Using that $Q_d$ is vertex-transitive, we may assume w.l.o.g.\ that the forbidden vertex is $z=\emptyset$.
The following theorem requires some additional conditions, and to state them we need to introduce some notation:
We split $Q_d$ in a \defi{direction}~$i \in [d]$ into subgraphs~$Q^i_0$ and $Q^i_1$, where $Q^i_1$ is the subgraph induced by vertices of~$Q_d$ containing~$i$ and $Q^i_0$ is the subgraph induced by vertices of~$Q_d$ not containing~$i$.
Note that $z \in V(Q^i_0)$ for every $i \in [d]$.
The set of edges of~$Q_d$ that have one end vertex in~$Q^i_0$ and the other in~$Q^i_1$ is called a \defi{layer} in direction~$i$.
This set splits into two equal-sized sets of edges, depending on the parity of their end vertices in~$Q^i_0$, and we refer to each of these two sets of edges as a \defi{half-layer}.
In this way, all $d2^{d-1}$ edges of~$Q_d$ are partitioned into $d$ many layers each of cardinality~$2^{d-1}$, or $2d$ many half-layers, each of cardinality~$2^{d-2}$.

\begin{theorem}
\label{thm:forbidden}
Let $d \ge 5$. A matching~$M$ of $K(Q_d)$ that avoids $z=\emptyset$ can be extended to a cycle that avoids~$z$ if and only if it satisfies the following condition:
\begin{itemize}[leftmargin=6mm, noitemsep, topsep=1pt plus 1pt]
\item[\propH{}] for every $i \in [d]$, if $M$ contains a half-layer in direction~$i$, then there is a vertex~$u\in V(Q^i_0) \setminus\{z\}$ that is avoided by~$M$.
\end{itemize}
\end{theorem}

We remark that the statement of Theorem~\ref{thm:forbidden} is true for $d=2$ and~$d=3$, but false for $d=4$, and all counterexamples can be seen in~\cite[Figs.~7--9]{MR3830138}.

In the following we will repeatedly refer to the condition on~$M$ stated in Theorem~\ref{thm:forbidden} as \defi{property~\propH{}}.
Note that this property depends on~$Q_d$, $M$ and~$z$, but for simplicity we do not make this dependence explicit, as the context should always be clear later when we apply the theorem to subcubes~$Q^i_0$ and~$Q^i_1$ for some~$i\in[d]$, certain matchings in those subcubes, and certain vertices to avoid in them.
Note that the vertex~$u$ must have the same (even) parity as~$z$, as all vertices in~$Q^i_0$ of the opposite (odd) parity are covered by the half-layer.

We prove Theorems~\ref{thm:cycle} and \ref{thm:forbidden} by induction on the dimension~$d$.
Specifically, in the induction step we assume that both statements hold in dimension~$d-1$, and we use this to prove that they also hold in dimension~$d$.

\subsection{Applications to Hamilton-laceability}

The hypercube~$Q_d$ is \defi{Hamilton-laceable}, i.e., it admits a Hamilton path between any two prescribed end vertices of opposite parity~\cite{MR527987}.
Gregor, Novotn\'{y}, and {\v{S}}krekovski~\cite{MR3830138} considered laceability combined with matching extensions.
Specifically, they considered the problem of extending a perfect matching of~$Q_d$ to a Hamilton path between two prescribed end vertices with opposite parity.
Their proof works in the more general setting of the complete bipartite graph~$B(Q_d)$.
For a matching~$M$ of~$Q_d$ and one of its vertices~$x$, we write $x^M$ for the other end vertex of the edge of~$M$ incident with~$x$.

\begin{theorem}[{\cite[Thm.~2]{MR3830138}}]
\label{thm:Hamlace1-BQd}
Let $d\ge 5$, and let $x, y$ be two vertices of opposite parity in~$Q_d$.
A perfect matching~$M$ of~$B(Q_d)$ with $xy\notin M$ can be extended to a Hamilton path with end vertices~$x$ and~$y$ if and only if $(M\setminus\{xx^M,yy^M\})\cup\{x^My^M\}$ contains no half-layers.
\end{theorem}

The following is an equivalent reformulation of this theorem.

\begin{theorem}[{\cite[Thm.~3]{MR3830138}}]
\label{thm:Hamlace2-BQd}
Let $d \ge 5$, and let $x, y$ be two vertices of opposite parity in~$Q_d$.
A perfect matching~$M$ of $B(Q_d \setminus\{x,y\})$ can be extended to a cycle that avoids~$x$ and~$y$ if and only if $M$ contains no half-layers.
\end{theorem}

In this theorem, $Q_d\setminus\{x,y\}$ represents the graph obtained from~$Q_d$ by removing the two vertices~$x$ and~$y$.

The authors also conjectured the strengthenings of Theorems~\ref{thm:Hamlace1-BQd} and~\ref{thm:Hamlace2-BQd} obtained by replacing $B(Q_d)$ and~$B(Q_d\setminus\{x,y\})$ by $K(Q_d)$ and $K(Q_d\setminus\{x,y\})$, respectively.
In fact, these stronger versions follow easily from our Theorem~\ref{thm:forbidden}, settling the conjecture raised by Gregor, Novotn\'{y}, and {\v{S}}krekovski.

\begin{theorem}
\label{thm:Hamlace1-KQd}
Let $d\ge 5$, and let $x, y$ be two vertices of opposite parity in~$Q_d$.
A perfect matching~$M$ of~$K(Q_d)$ with $xy\notin M$ can be extended to a Hamilton path with end vertices~$x$ and~$y$ if and only if $(M\setminus\{xx^M,yy^M\})\cup\{x^My^M\}$ contains no half-layers.
\end{theorem}

We prove the following equivalent reformulation of this theorem.

\begin{theorem}
\label{thm:Hamlace2-KQd}
Let $d \ge 5$, and let $x, y$ be two vertices of opposite parity in~$Q_d$.
A perfect matching~$M$ of $K(Q_d \setminus\{x, y\})$ can be extended to a cycle that avoids~$x$ and~$y$ if and only if $M$ contains no half-layers.
\end{theorem}

\subsection{Avoiding and including other structures}

\begin{figure}
\centerline{
\includegraphics[page=2]{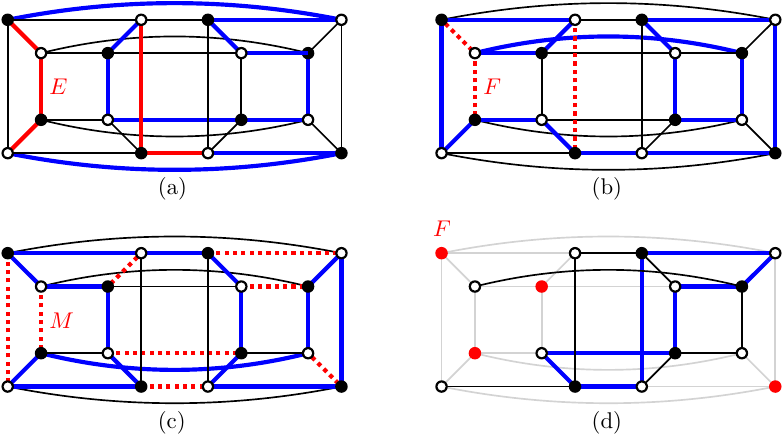}
}
\caption{Illustration of Hamilton cycles in~$Q_d$, $d=4$, subject to various constraints:
(a)~a set~$E$ of $2d-3=5$ edges and a Hamilton cycle extending it;
(b)~a set~$F$ of $2d-5=3$ edges and a Hamilton cycle avoiding it;
(c)~a perfect matching~$M$ and a Hamilton cycle avoiding it;
(d)~a set~$F$ of $\binom{d}{2}-2=4$ vertices and a cycle of length~$2^d-2|F|=8$ avoiding it.}
\label{fig:ham}
\end{figure}

In the literature, there has also been substantial work on extending sets of edges~$E$ other than matchings to Hamilton cycles.
In addition, there is interest in avoiding certain sets~$F$ of `forbidden' edges.
These questions are motivated by applications in computer networks, where the hypercube is frequently used as a network topology with a number of desirable properties, such as small degree and diameter.
In this context, including or avoiding certain edges corresponds to prescribed connections or faulty connections, respectively.

Dvo\v{r}\'{a}k~\cite{MR2178189} showed that for any set~$E$ of at most~$2d-3$ edges in~$Q_d$, $d\ge 2$, that form disjoint paths, there is a Hamilton cycle that contains all of~$E$; see Figure~\ref{fig:ham}~(a).
Dvo\v{r}\'{a}k and Gregor~\cite{MR2326160} proved that for any set~$E$ of at most~$2d-4$ edges in~$Q_d$, $d\ge 5$, that form disjoint paths and two vertices~$x$ and~$y$ of opposite parity that are neither internal vertices of the paths nor end vertices of the same path, there is a Hamilton path with end vertices~$x$ and~$y$ that contains all of~$E$.

We now consider the problem of avoiding a set~$F$ of `forbidden' edges.
In this direction, Latifi, Zheng and Bagherzadeh~\cite{DBLP:conf/ftcs/LatifiZB92} showed that for any set~$F$ of at most~$d-2$ edges in~$Q_d$, there is a Hamilton cycle that avoids~$F$.
Chan and Lee~\cite{MR1129389} proved that for any set~$F$ of at most~$2d-5$ edges in~$Q_d$, $d\ge 3$, such that every vertex of~$Q_d$ is incident to at least two edges not in~$F$, there is a Hamilton cycle that avoids~$F$ (see also~\cite{MR3123444}); see Figure~\ref{fig:ham}~(b).
With regards to perfect matchings, Dimitrov, Dvo\v{r}\'{a}k, Gregor and \v{S}krekovski~\cite{MR2568029} proved that for a given perfect matching~$M$ of~$Q_d$, there is a Hamilton cycle that avoids~$M$ if and only if~$Q_d\setminus M$ is a connected graph; see Figure~\ref{fig:ham}~(c).

Interestingly, the bounds $2d-3$, $2d-4$, $d-2$, $2d-5$ in the aforementioned results on extending or avoiding certain sets of edges are all best possible, i.e., prescribing or forbidding more edges sometimes leads to situations where the desired Hamilton cycle or path does not exist.

Instead of forbidding edges, one may also consider the problem of forbidding certain vertices.
To this end, Fink and Gregor~\cite{MR2970496} proved that for any set~$F$ of at most~$\binom{d}{2}-2$ vertices of~$Q_d$, $d\ge 3$, there is a cycle of length at least~$2^d-2|F|$ that avoids all vertices in~$F$; see Figure~\ref{fig:ham}~(d).
Their result answers a conjecture by Casta\~{n}eda and Gotchev~\cite{MR2725517}, and the bound~$\binom{d}{2}-2$ is best possible.
Note that if we forbid some number of vertices of the same parity, then any extending cycle will also skip the same number of vertices of the opposite parity.
This explains that we need to allow skipping~$2|F|$ many vertices instead of only~$|F|$.

\subsection{Other restricted Gray codes}

Hamilton cycles and paths in the hypercube are often referred to as \defi{Gray codes}~\cite{MR4649606,MR1491049}.
In this setting, we view the vertices of the hypercube as bitstrings of length~$n$, given by the indicator vectors of the corresponding subsets.
The well-known binary reflected Gray code, named after the Bell Labs researcher Frank Gray, is a particularly elegant construction of such a cycle that lends itself well to fast algorithmic computation, optimally so that each bit flipped is computed in constant time.
In addition to the work discussed in the preceding section about Gray codes that include or avoid certain substructures, there has been a large amount of work on Gray codes that satisfy various other kinds of constraints, motivated by different applications:
\begin{itemize}[leftmargin=6mm, noitemsep, topsep=1pt plus 1pt]
\item balanced Gray codes~\cite{MR1410880}, where each bit is flipped equally often;
\item non-local Gray codes~\cite{MR1081842}, where any $2^t$ consecutive bitstrings flip more than $t$ distinct bits, for every $2\leq t\leq n-1$;
\item run-length restricted Gray codes~\cite{MR2014514}, where any two consecutive flips of the same bit are almost $n$ steps apart;
\item antipodal Gray codes~\cite{MR2047769}, where the complement of any bitstring~$x$ is visited exactly $n$ steps before or after~$x$;
\item monotone Gray codes~\cite{MR1329390}, where the Hamming weight of the bitstrings never decreases by more than~2 from what it was before along the code;
\item transition restricted Gray codes~\cite{MR1377601,MR2974271}, where the pairs of consecutively flipped bits are edges of a certain transition graph;
\item Beckett-Gray codes~\cite{DBLP:journals/endm/SawadaW07}, where on each transition $1\rightarrow 0$ the least recently entered 1-bit has to be flipped;
\item single-track Gray codes~\cite{MR2629553,MR4747482}, where each bit follows the same cyclically shifted flipping pattern along the code.
\end{itemize}

\subsection{Outline of this paper}

In Section~\ref{sec:prelim} we present some notations and terminology used throughout this paper, as well as some auxiliary lemmas needed later.
In Sections~\ref{sec:cycle} and~\ref{sec:forbidden} we present the inductive arguments for the proofs of Theorems~\ref{thm:cycle} and~\ref{thm:forbidden}, respectively.
Section~\ref{sec:other-proofs} contains the proofs of Theorems~\ref{thm:long-Qd}, \ref{thm:long-KQd}, and~\ref{thm:Hamlace2-KQd}.
We conclude with some open problems in Section~\ref{sec:open}.

\section{Preliminaries}
\label{sec:prelim}

We need the following definitions and auxiliary lemmas.

\subsection{Notation and definitions}

For a graph~$G$, the sets of vertices and edges of~$G$ are denoted by~$V(G)$ and~$E(G)$, respectively.
Furthermore, we write $K(G)$ for the complete graph on the vertex set of~$G$.
If~$G$ is connected and bipartite, then there is a unique bipartition of its vertices, and we write~$B(G)$ for the complete bipartite graph with that bipartition.
In other words, $K(G)$ is obtained from~$G$ by adding all missing edges to~$G$, and $B(G)$ is obtained from~$G$ by adding all possible edges to~$G$, while preserving the property that the graph remains bipartite.

For two sets~$A$ and~$B$, the symmetric difference of~$A$ and~$B$ is denoted by~$A \symdiff B$.
Note that $Q_d$ has an edge between any two vertices~$u$ and~$v$ with $|u\symdiff v|=1$.
For a vertex $u$ of~$Q_d$ and $i \in [d]$, we define $u^i := u \symdiff \{i\}$, and we refer to~$i$ as the \defi{direction} of the edge~$uu^i$.
We say that~$u$ is \defi{even} or \defi{odd}, if $|u|$ is even or odd, respectively.

A \defi{linear forest} is collection of vertex-disjoint paths.
The \defi{terminals} of a linear forest are its vertices of degree~1, i.e., the end vertices of the paths.
A \defi{shortcut} of a linear forest in a graph~$G$ is the set of edges of~$K(G)$ that connect the two terminals of every path to each other.

For a graph~$G$ and a subset of edges~$M$ of~$G$, we say that a vertex~$u$ of~$G$ is \defi{covered} by~$M$, if $u$ is an end vertex of one of the edges of~$M$.
We write $V(M)$ for the set of all vertices covered by~$M$.
If $u$ is not covered by~$M$ we say that it is \defi{avoided} by~$M$.
Recall that for a vertex~$u$ covered by an edge of a matching~$M$, we write $u^M$ for the other end vertex of that edge of~$M$.
Furthermore, for $i\in[d]$ we use the abbreviation $u^{iM} := (u^i)^M$.

For a set of edges $F \subseteq E(K(Q_d))$ and a direction~$i\in[d]$ we define $F^i_0 := F \cap E(K(Q^i_0))$ and $F^i_1 := F \cap E(K(Q^i_1))$, and we write $F^i_-$ for the edges of~$F$ joining vertices of~$Q^i_0$ and~$Q^i_1$.
This partitions $F$ into three disjoint sets~$F^i_0$, $F^i_1$ and~$F^i_-$.

The following definitions are illustrated in Figure~\ref{fig:layer}.
Recall the definition of half-layers given in Section~\ref{sec:results}.
Given some direction~$i\in[d]$, a half-layer in $Q^i_0$ or in~$Q^i_1$ is referred to as a \defi{quad-layer} in~$Q_d$.
Clearly, a half-layer has $2^{d-2}$ edges, and a quad-layer has $2^{d-3}$ edges.
A \defi{near half-layer} (\defi{near quad-layer}) is a half-layer (quad-layer) minus one edge.
We refer to the two end vertices of the removed edge, which are not part of the near half-layer (near quad-layer), as \defi{extension vertices} of the near half-layer (near quad-layer).
A \defi{2-near half-layer} is a half-layer minus two edges.
We refer to the four end vertices of the removed edges, which are not part of the 2-near half-layer, as \defi{extension vertices} of the 2-near half layer.
Given a matching~$M$ that includes a near half-layer (near quad-layer), we say that this near half-layer (near quad-layer) is \defi{covered} if both of its extension vertices are covered by two distinct edges of~$M$.

\begin{figure}
\centerline{
\includegraphics[page=1]{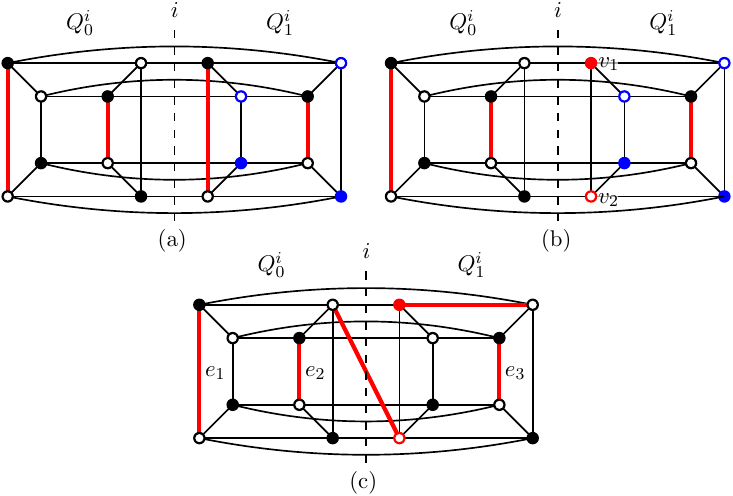}
}
\caption{The red edges in~$Q_4$ form a (a) half-layer, (b) near half-layer with extension vertices~$v_1$ and~$v_2$, or (c) matching that includes the covered near half-layer~$\{e_1,e_2,e_3\}$.
The subset of red edges incident with a vertex in~$Q^i_1$ form a (a) quad-layer, (b) near quad-layer, or (c) matching that includes the covered near quad-layer~$\{e_3\}$.
Every vertex~$x$ in (a) and~(b) colored blue is one for which the (near) quad-layer contained in~$Q^i_1$ is $x$-dangerous.
The dashed vertical line separates~$Q^i_0$ and~$Q^i_1$.}
\label{fig:layer}
\end{figure}

A half-layer of~$Q_d$ is \defi{$x$-dangerous} for every vertex~$x$ not incident with any edge from the half-layer.
An \defi{$x$-dangerous} near half-layer is an $x$-dangerous half-layer minus one edge.
Note that $x$ cannot be one of the extension vertices of the near half-layer.
For a vertex~$x$ of~$Q_d$, a quad-layer is \defi{$x$-dangerous} if $x$ is not incident with any edge from the quad-layer and $x$ belongs to the same $(d-1)$-dimensional subcube as the quad-layer.
Furthermore, we define an \defi{$x$-dangerous} near quad-layer as an $x$-dangerous quad-layer minus one edge.
Similarly to before, $x$ cannot be one of the extension vertices of the near quad-layer.

\subsection{A cycle that avoids two vertices}

The first lemma asserts that a cycle that extends an almost perfect matching of~$K(Q_d)$ and avoids a particular vertex~$x$ is forced to also avoid another vertex~$y$, provided that $x$ and $y$ have the correct parity.
This is useful when applying induction in the proof of Theorem~\ref{thm:forbidden}, because it prevents that the invariant `cycle avoids one prescribed vertex' blows up to `cycle avoids two prescribed vertices', `cycle avoids three prescribed vertices', etc., which would be undesirable.

\begin{lemma}
\label{lem:avoid2}
Let $x,y$ be two vertices of~$Q_d$ and let $M$ be a perfect matching of~$K(Q_d \setminus\{x,y\})$.
Let $C$ be a cycle that extends~$M$ and avoids~$x$.
Then, $C$ avoids~$y$ if and only if $x$ and~$y$ have opposite parity.
\end{lemma}

\begin{proof}
If $C$ avoids~$y$, then $C\setminus M$ is set of pairwise disjoint edges whose end vertices have opposite parity (i.e., a matching), covering all vertices of~$K(Q_d\setminus\{x,y\})$.
Since $Q_d$ has the same number of even and odd vertices, it follows that the avoided vertices~$x$ and~$y$ also have opposite parity.
On the other hand, if $C$ contains~$y$, then $C\setminus M$ is a set of pairwise disjoint edges whose end vertices have opposite parity, plus a path of length~2 whose two end vertices have opposite parity to the middle vertex~$y$.
It follows that~$x$ must have the same parity as~$y$.
\end{proof}

\subsection{Properties of half-layers and quad-layers}

Our first lemma describes the union of two or more half-layers in~$Q_d$.

\begin{lemma}
\label{lem:hlayers-structure}
Half-layers in~$Q_d$ satisfy the following properties:
\begin{enumerate}[label=(\roman*),leftmargin=8mm, noitemsep, topsep=1pt plus 1pt]
\item The union of two half-layers~$L$ and~$L'$ in different directions is a collection of paths of length~2, i.e., every edge of~$L$ is incident with exactly one edge of~$L'$ and vice versa, and they share exactly~$2^{d-2}$ vertices.
\item The union of $k$ half-layers in $k$ different directions avoids exactly~$2^{d-k}$ vertices of~$Q_d$.
\end{enumerate}
\end{lemma}

\begin{proof}
To prove~(i), let $i$ and~$j$ be the directions of~$L$ and~$L'$, respectively, and consider the partition of~$Q_d$ into disjoint 4-cycles of the form~$C=(u,u^i,(u^i)^j,u^j)$ for some~$u\in V(Q_d)$.
From the two pairs of opposite edges in directions~$i$ and~$j$ in~$C$, exactly one edge belongs to~$L$ and one edge to~$L'$, respectively.

To prove~(ii), note that the vertices~$x$ of~$Q_d$ that are avoided by a half-layer in direction~$j$ are exactly those that satisfy the equation~$\sum_{i\in[d]\setminus\{j\}}x_i=a\pmod{2}$ for some $a\in\{0,1\}$.
Consequently, the vertices that are avoided simultaneously by $k$ half-layers in $k$ different directions satisfy $k$ linearly independent equations of this form (over~$\mathbb{Z}_2$), so the number of solutions is~$2^{d-k}$.
\end{proof}

The next lemma asserts that a matching cannot contain half-layers and quad-layers in several different directions.

\begin{lemma}
\label{lem:hlayers-number}
Let $M$ be a matching of~$K(Q_d)$ and let $z=\emptyset$.
\begin{enumerate}[label=(\roman*),leftmargin=8mm, noitemsep, topsep=1pt plus 1pt]
\item If $d \ge 4$, then $M$ contains a (near) half-layer in at most one direction.
\item If $d \ge 6$, then $M$ contains a $z$-dangerous (near) quad-layer in at most one direction.
\item If $d \ge 5$, then $M$ contains a 2-near half-layer in at most one direction.
\end{enumerate}
\end{lemma}

We remark that the lower bounds on the dimension~$d$ stated in Lemma~\ref{lem:hlayers-number} are all best possible.

\begin{proof}
The first two statements are proved in~\cite{MR3830138}, specifically in Lemma~6~(i) and Lemma~7~(vi).

To prove~(iii) we argue as follows.
By Lemma~\ref{lem:hlayers-structure}~(i), two half-layers in different directions share $2^{d-2}$ vertices.
Having two 2-near half-layers in different directions, each one uses $2^{d-2}-2$ of the shared vertices to be matched in its direction (and not the other direction).
However, $2(2^{d-2}-2) > 2^{d-2}$ for $d\ge 5$, so two 2-near half-layers in different directions cannot exist for $d \ge 5$.
\end{proof}

\begin{lemma}
\label{lem:Q5-quad-cut}
Let $d=5$, and let $M$ be a matching of~$K(Q_d)$ that avoids~$z=\emptyset$ and contains a $z$-dangerous (near) quad-layer.
Then there is a direction~$i\in[d]$ such that $|M^i_-|\ge 3$ and $M^i_0$ contains no (near) half-layers of~$Q^i_0$.
\end{lemma}

\begin{proof}
Let $i_1$ be the direction of a $z$-dangerous (near) quad-layer, implying that $|M^{i_1}_-|\ge 2^{d-3}-1=3$.
If $M^{i_1}_0$ contains no (near) half-layers, then we are done.
Otherwise, let $i_2$ be the direction of a (near) half-layer of~$Q^{i_1}_0$, implying that $|M^{i_2}_-|\ge 3$.
If $M^{i_2}_0$ contains no (near) half-layers, then we are done.
Otherwise, let $i_3$ be the direction of a (near) half-layer of~$Q^{i_2}_0$, implying that $|M^{i_3}_-|\ge 3$ etc.
If this process terminates, then we are done.
Otherwise, it loops in a cycle of length~$\ell \ge 2$.
As each of these $\ell$ many $z$-dangerous (near) quad-layers of~$Q_d$ in direction~$i_j$, $j\in\{1,\ldots,\ell\}$, covers at least $d-3=2$ neighbors of~$z$ with an edge in direction~$i_j$, we conclude that $\ell=2$.
For $\ell=2$, we have a (near) half-layer of $Q^{i_1}_0$ in direction~$i_2$ and a (near) half-layer of $Q^{i_2}_0$ in direction~$i_1$, and each of them covers at least $d-3=2$ neighbors of~$z$ in the subcube~$(Q^{i_1}_0)^{i_2}_0$ with an edge in direction~$i_2$ and~$i_1$, respectively.
However, the vertex $z$ has only three neighbors in the 3-dimensional subcube~$(Q^{i_1}_0)^{i_2}_0$.
\end{proof}

The next lemma allows us to complete partial matchings of~$K(Q_d)$ to perfect matchings without creating half-layers.

\begin{lemma}
\label{lem:avoid-layer}
Let $d\ge 4$, let $M$ be a matching of~$K(Q_d)$ that contains no half-layers, and let $A \subseteq V(Q_d) \setminus V(M)$ such that $|A| \ge 4$ and even.
Then there is a perfect matching~$P$ on~$K(A)$ such that $M \cup P$ contains no half-layers.
The same statement holds for near half-layers.
\end{lemma}

\begin{proof}
We first prove the statement for half-layers.
We repeatedly match pairs of vertices in~$A$, adding the matched pairs to~$P$, thus decreasing the size of~$A$ in each round, as follows:
If $|A| \ge 6$, we match any two vertices of~$A$ having the same parity, which does not create half-layers, and repeat.
If $|A| = 4$ we distinguish three cases.
If all vertices in~$A$ have the same parity, we match them arbitrarily, which does not create half-layers.
If $A$ has two vertices of each parity, we match them in those pairs, and we are done as well.
Otherwise we have one even and three odd vertices in~$A$ or vice versa, and by symmetry it suffices to consider the first case.
Let $a$ be the unique even vertex in~$A$.
If one of the three odd vertices in~$A$ is not a neighbor of~$a$ in~$Q_d$, then we match $a$ to this vertex and the remaining two vertices to each other, which does not create half-layers.
On the other hand, if all three odd vertices in~$A$ are neighbors of~$a$, then for at most two of them adding an edge to~$a$ may create a half-layer.
Indeed, if for all three of them adding an edge to~$a$ would create a half-layer, then this would imply $|M|\geq 3(2^{d-2}-1)=:\ell$, so $M$ contains $\ell$ odd vertices, which together with $A$ are $\ell+3\geq 3\cdot 2^{d-2}>2^{d-1}$ odd vertices, but $Q_d$ has only $2^{d-1}$ odd vertices in total.
It follows that we can add an edge from~$a$ to one of the vertices in~$A$ without creating a half-layer, and then the remaining two vertices in~$A$ can be matched to each other.

The same proof works for near half-layers instead of half-layers, the only difference being the inequalities $|M|\geq 3(2^{d-2}-2)=:\ell$ and $\ell+3\geq 3(2^{d-2}-1)>2^{d-1}$ in the last step that are valid for $d\geq 4$.
\end{proof}

\subsection{Large cuts through maximal matchings}

For any edge $e=uv$ of~$K(Q_d)$ we define the \defi{length} of~$e$ to be the quantity $\ell(uv):=|u\symdiff v|$.
If $\ell(e)\ge 2$ then we say that $e$ is \defi{long}.
In other words, long edges are edges of~$K(Q_d)$ that are not present in~$Q_d$.

For a set of edges $F \subseteq E(K(Q_d))$ we define
\begin{equation}
\label{eq:ellF}
\ell(F):=\sum_{e\in F}\ell(e) = \sum_{i \in [d]} |F^i_-|.
\end{equation}

We say that a matching~$M$ of~$K(Q_d)$ is \defi{maximal} if~$M$ covers at least one end vertex of every edge of~$Q_d$.

\begin{lemma}
\label{lem:matching-length}
For every maximal matching~$M$ of~$K(Q_d)$ there is a maximal matching~$M'$ of~$Q_d$ such that~$\ell(M) \ge \ell(M')$.
\end{lemma}

\begin{proof}
Suppose that $M$ contains a long edge~$uv$, i.e., $\ell(uv)\ge 2$.
We define~$M':=M\setminus\{uv\}$.
Furthermore, if there is a neighbor~$u'$ of~$u$ in~$Q_d$ that is not covered by~$M'$, then we redefine $M':=M'\cup\{uu'\}$.
Similarly, if there is a neighbor~$v'$ of~$v$ in~$Q_d$ that is not covered by~$M'$, then we redefine $M':=M'\cup\{vv'\}$.
As an edge of length $\ell(uv)\ge 2$ is removed from~$M$ and at most two length~1 edges are added instead, we have $\ell(M)\ge \ell(M')$.
Furthermore, note that $M'$ is again a maximal matching of~$K(Q_d)$ with one less long edge than~$M$.
Consequently, repeating this replacement process terminates with a maximal matching of~$Q_d$.
This completes the proof of the lemma.
\end{proof}

The next two lemmas establish lower bounds for the size of maximal matchings of~$Q_d$ and~$K(Q_d)$, respectively.

\begin{lemma}[{\cite[Lemma~3]{MR321804}}]
\label{lem:max-Qd}
Every maximal matching of~$Q_d$ has at least $f(d):=\frac{d}{3d-1}|V(Q_d)|=\frac{d 2^d}{3d-1}$ edges.
\end{lemma}

The function~$f(d)$ defined in Lemma~\ref{lem:max-Qd} will be used in several computations in the rest of this section.
For dimension $d=6$, Havel and K\v{r}iv\'{a}nek~\cite{MR653356} improved the lower bound of $\lceil f(6)\rceil=23$ guaranteed by Lemma~\ref{lem:max-Qd} to~$24$, but we do not need this strengthening for our arguments.

\begin{lemma}
\label{lem:max-KQd}
Every maximal matching of~$K(Q_d)$ has at least $\frac{1}{4}|V(Q_d)|=2^{d-2}$ edges.
\end{lemma}

Note that this lower bound is attained by any matching that covers all vertices of the same parity, and none of the opposite parity.

\begin{proof}
Let $M$ be a maximal matching of~$K(Q_d)$, and consider the layer~$L$ in direction~1, which satisfies $|L|=\frac{1}{2}|V(Q_d)|=2^{d-1}$.
$M$ covers at least one of the end vertices of every edge of~$L$, so we have $|M|\ge |L|/2=\frac{1}{4}|V(Q_d)|=2^{d-2}$.
\end{proof}

\begin{lemma}
\label{lem:maximal-cut}
For every maximal matching~$M$ of~$K(Q_d)$ there is a direction~$i \in [d]$ such that $|M^i_-| \ge 3$ for $d = 5$ and $|M^i_-| \ge 4$ for $d \ge 6$.
\end{lemma}

\begin{proof}
Let $M$ be a maximal matching of~$K(Q_d)$.
By Lemma~\ref{lem:matching-length} there is a maximal matching~$M'$ of~$Q_d$ with $\ell(M)\ge \ell(M')=|M'|$.
Furthermore, Lemma~\ref{lem:max-Qd} yields that $|M'|\ge f(d)$.
Combining these two observations proves $\ell(M)\ge f(d)$.
From~\eqref{eq:ellF} we obtain that there is a direction~$i\in[d]$ with $|M^i_-|\ge \lceil f(d)/d\rceil$.
For $d=5$ and $d\ge 6$ the function~$\lceil f(d)/d\rceil$ evaluates to~3 and~$\ge 4$, respectively.
\end{proof}

The next lemma describes how adding an edge to a matching of~$K(Q_d)$ can violate property~\propH{}. 

\begin{lemma}
\label{lem:H-violation}
Let $d\geq 5$, and let $M$ be matching of $K(Q_d)$ that avoids $z=\emptyset$ and satisfies property~\propH{} in Theorem~\ref{thm:forbidden}.
Furthermore, let $i\in [d]$ and $u,u^i$ with $u\in V(Q^i_0)\setminus\{z\}$ be two vertices avoided by~$M$.
Then $M \cup \{uu^i\}$ violates property~\propH{} if and only if one of the following two conditions holds:
\begin{enumerate}[label=(\roman*),leftmargin=8mm, noitemsep, topsep=1pt plus 1pt]
\item $M$ contains a half-layer in direction~$i$ and all vertices of~$Q^i_0$ except~$z$ and~$u$ are covered by~$M$;
\item $M$ contains a near half-layer in direction~$i$ with extension vertices~$u$ and~$u^i$, and all vertices of~$Q^i_0$ except~$z$ and~$u$ are covered by~$M$.
\end{enumerate}
In both cases $|M^i_-|$ is even.
Furthermore, we have $|M^i_-|\ge \frac{1}{4}|V(Q_d)|=2^{d-2}$ and $|M|\ge \frac{3}{8}|V(Q_d)|-1=3\cdot 2^{d-3}-1$.
\end{lemma}

\begin{proof}
Statements~(i) and~(ii) follow directly from the definition of property~\propH{}.
The next claim follows from the fact that $V(Q^i_0)\setminus\{z,u\}$ has even cardinality, and every edge in~$M^i_-$ covers exactly one of the vertices from this set, whereas every edge in~$M^i_0$ covers two of them.
To prove the last part we argue as follows:
If (i) holds, then $M$ contains $2^{d-2}$ edges in the half-layer in direction~$i$, which all belong to~$M^i_-$, plus at least $(2^{d-2}-2)/2=2^{d-3}-1$ edges that cover all remaining vertices in~$Q^i_0$ except~$z$ and~$u$, which in total gives at least $3\cdot 2^{d-3}-1$ edges in~$M$.
If (ii) holds, then $M$ contains $2^{d-2}-1$ edges in the near half-layer in direction~$i$, which all belong to~$M^i_-$, plus at least $\lceil (2^{d-2}-1)/2\rceil=2^{d-3}$ edges that cover all remaining vertices in~$Q^i_0$ except~$z$ and~$u$, at least one of which belongs to~$M^i_-$, which in total gives at least $2^{d-2}$ edges in~$M^i_-$ and at least $3\cdot 2^{d-3}-1$ edges in~$M$.
\end{proof}

We say that a matching~$M$ of~$K(Q_d)$ that avoids~$z=\emptyset$ and satisfies property~\propH{} is \defi{\propH{}-maximal} if for any two vertices $u,u^i$ for some~$i\in[d]$ that are avoided by~$M$ and different from~$z$ the matching $M \cup \{uu^i\}$ violates property~\propH{}.
Note that any maximal matching that avoids~$z$ is also~\propH{}-maximal.

The next lemma guarantees us a direction such that many edges of an \propH{}-maximal matching belong to the layer in this direction.
In the proof of Theorem~\ref{thm:forbidden} we will choose this direction for splitting the cube into two subcubes and applying induction.

\begin{lemma}
\label{lem:H-maximal}
For every \propH{}-maximal matching~$M$ of~$K(Q_d)$ there is a direction~$i\in [d]$ such that~$|M^i_-| \ge 3$ for $d = 5$ and~$|M^i_-| \ge 4$ for $d\ge 6$.
\end{lemma}

\begin{proof}
There are three cases to consider.

The first case is that there are vertices~$u,u^i$ with $i\in[d]$ that are avoided by~$M$ and different from~$z$.
In this case $M\cup\{uu^i\}$ violates property~\propH{}, and applying Lemma~\ref{lem:H-violation} yields $|M^i_-|\ge 2^{d-2}$.
For $d=5$ and $d\ge 6$ this lower bound evaluates to~8 and at least~16, respectively.

The second case is that $M$ is maximal, and then the claim follows from Lemma~\ref{lem:maximal-cut}.

The third case is that one of the neighbors~$u$ of~$z$ in~$Q_d$ is avoided by~$M$ and $M\cup\{uz\}$ is maximal.
In this case, by Lemmas~\ref{lem:matching-length} and~\ref{lem:max-Qd}, we obtain~$\ell(M)\ge f(d)-1$, and consequently there is a direction~$i\in[d]$ with $|M^i_-|\ge \lceil (f(d)-1)/d\rceil$.
For $d=5$ and $d\ge 6$ the function~$\lceil (f(d)-1)/d\rceil$ evaluates to~3 and at least~4, respectively.

The claimed lower bounds hold in all three cases, which proves the lemma.
\end{proof}

\section{Proof of Theorem~\ref{thm:cycle}}
\label{sec:cycle}

In this section, we prove Theorem~\ref{thm:cycle}.

\begin{proof}[Proof of Theorem~\ref{thm:cycle}]
For $2 \le d \le 4$, we verified the theorem with computer help.
For details, see Section \ref{sec:basis}.

For $d\ge 5$, we prove Theorem~\ref{thm:cycle} assuming that Theorem~\ref{thm:forbidden} holds for dimension~$d$.

If $M$ is perfect, we apply Theorem~\ref{thm:kreweras}.
Otherwise, if $M$ has no half-layers in any direction, then we choose a vertex~$u$ avoided by~$M$.
Otherwise, if $M$ contains a half-layer in a direction~$i\in[d]$ such that $M$ avoids at least two vertices of~$Q^i_0$ or~$Q^i_1$, then we choose a vertex~$u$ avoided by~$M$ in~$Q^i_0$ or~$Q^i_1$, respectively.
By Lemma~\ref{lem:hlayers-number}~(i) and the fact that $M$ is not perfect, $M$ contains no other half-layer.
Consequently, in the latter two cases, $M$ and~$u$ satisfy property~\propH{} (where $u$ plays the role of~$z=\emptyset$).
We may thus apply Theorem~\ref{thm:forbidden} to obtain a cycle that extends~$M$ and avoids~$u$.

In the remaining case, $M$ contains a half-layer~$L$ in a direction~$i\in[d]$ and avoids exactly one vertex~$u$ in $Q^i_0$ and
one vertex~$v$ in~$Q^i_1$.
If $v=u^i$, we apply Theorem~\ref{thm:kreweras} to the perfect matching $M \cup \{uv\}$.
Otherwise consider the modified matching~$N:=(M \setminus\{u^i u^{iM}\}) \cup \{u u^{iM}\}$, which avoids two vertices in~$Q^i_1$, namely $u^i$ and~$v$.
By Lemma~\ref{lem:hlayers-number}~(i) and the fact that $N$ is not perfect, $N$ contains no half-layer other than~$L$.
Consequently, $N$ and~$u^i$ satisfy property~\propH{} (where $u^i$ plays the role of~$z=\emptyset$), so applying Theorem~\ref{thm:forbidden} gives a cycle~$C$ that extends~$N$ and avoids~$u^i$.
Then, the cycle $(C \setminus\{u u^{iM}\}) \cup \{u u^i, u^i u^{iM}\}$ extends~$M$.
\end{proof}

\section{Proof of Theorem~\ref{thm:forbidden}}
\label{sec:forbidden}

The statement in Theorem~\ref{thm:forbidden} is an equivalence, and we prove both directions of the implication in the following subsections.
We first consider the forward implication, then the reverse implication for $d\geq 6$, and then the reverse implication for $d=5$ (settled with computer help), followed by some remarks about implementation details for the computer verification.

\subsection{Forward implication}
\label{sec:forward}

\begin{proof}[Proof of Theorem~\ref{thm:forbidden} ($\Rightarrow$)]
For the sake of contradiction suppose that $M$ violates property~\propH{} and that there is a cycle~$C$ that extends~$M$ and avoids~$z$.
Then for some $i\in [d]$ there is a half-layer in direction~$i$ and all vertices of~$Q^i_0$ except~$z$ are covered by~$M$.
Every odd vertex of~$Q^i_0$ is an end vertex of an edge of the half-layer, and is therefore connected in~$C\setminus M$ to an even vertex of~$Q^i_0$.
As every vertex of~$Q^i_0$ except~$z$ is covered by~$M$, no two of these even vertices are the same.
It follows that~$C$ visits all odd vertices of~$Q^i_0$ and at least as many even vertices, but at the same time it avoids the even vertex~$z$, a contradiction.
\end{proof}

\subsection{Reverse implication (induction step \texorpdfstring{$d-1\rightarrow d$}{d-1 to d} for \texorpdfstring{$d\ge 6$}{d>=6})}

Our proof for the reverse implication in Theorem~\ref{thm:forbidden} uses induction on the dimension~$d$, and it follows a similar strategy as Fink's proof of Theorem~\ref{thm:kreweras}, namely to choose a direction~$i\in[d]$ and to split~$Q_d$ into subcubes~$Q^i_0$ and~$Q^i_1$, to which we apply induction.
However, the requirement for the extending cycle to avoid a prescribed vertex causes substantial additional technical complications.
In particular, we need to deal with the case that an odd number of edges is present in~$M^i_-$, a case that never occurs if the matching~$M$ is perfect.
In that case, we need to add one additional edge in direction~$i$ to be included in the extending cycle (as the cycle has to cross between~$Q^i_0$ and~$Q^i_1$ an even number of times).
We thus need to choose the direction~$i$ and the added edge so that no obstacles for applying induction in the subcubes, namely half-layers of~$Q^i_0$ or~$Q^i_1$, are created.

We first present the induction step~$d-1\rightarrow d$ for $d\ge 6$, whereas the base case~$d=5$ is settled in the next section.

\begin{proof}[Proof of Theorem~\ref{thm:forbidden} ($\Leftarrow$) for $d\ge 6$]
We prove Theorem~\ref{thm:forbidden} inductively for $d\ge 6$, assuming that Theorems~\ref{thm:cycle} and~\ref{thm:forbidden} hold for dimension~$d-1$.

Let~$M$ be a matching of~$K(Q_d)$ that avoids~$z$ and satisfies property~\propH{}.
We assume w.l.o.g.\ that $M$ is \propH{}-maximal (recall the definition from before Lemma~\ref{lem:H-maximal}).

We select a direction~$i\in[d]$ according to the following rules applied in this order:
\begin{enumerate}[leftmargin=8mm, noitemsep, topsep=1pt plus 1pt]
\item If $M$ contains a $z$-dangerous (covered near) quad-layer, we choose $i$ to be its direction;
\item otherwise, if $M$ contains a quad-layer, we choose $i$ to be its direction;
\item otherwise, we choose a direction $i \in [d]$ that maximizes the quantity~$|M^i_-|$.
\end{enumerate}
These rules guarantee the following properties:
\begin{enumerate}[label=(\roman*),leftmargin=8mm, noitemsep, topsep=1pt plus 1pt]
\item $M^i_0$ contains no (covered near) half-layers of~$Q^i_0$.
\item If $M^i_1$ contains a half-layer of~$Q^i_1$, then $M^i_-$ contains a (covered near) quad-layer of~$Q_d$.
\item If rule~(1) or~(2) applies we have~$|M^i_-|\ge 7$, and if rule~(3) applies we have~$|M^i_-|\ge 4$.
\end{enumerate}

Proof of~(i):
If $M^i_0$ did contain a (covered near) half-layer of~$Q^i_0$, then there would be a $z$-dangerous (covered near) quad-layer of~$Q_d$ contained in~$M$, and by Lemma~\ref{lem:hlayers-number}~(ii) these can occur in at most one direction, which would have been selected by rule~(1) with the highest priority.

Proof of~(ii):
This situation is illustrated in Figure~\ref{fig:proof2}.
Suppose that $M^i_1$ contains a half-layer~$L$ of~$Q^i_1$.
The direction~$i$ was chosen differently from the direction of~$L$, which means that rule~(1) or~(2) were applied, and this proves the claim.

Proof of~(iii):
The first part follows as a near quad-layer of~$Q_d$, $d\ge 6$, has $2^{d-3}-1\ge 7$ edges.
The second part follows from Lemma~\ref{lem:H-maximal}, using the assumption that~$M$ is \propH{}-maximal.

Let $A_0:=V(Q^i_0) \cap V(M^i_-)$.
From~(iii) we obtain $|A_0|\ge 4$.

\vspace{1ex}
\textbf{Case~1:} $|M^i_-|\geq 4$ is even.
By~(i), we can apply Lemma~\ref{lem:avoid-layer} and obtain a perfect matching~$P_0$ on~$K(A_0)$ such that $M^i_0 \cup P_0$ contains no half-layers.
Applying Theorem~\ref{thm:forbidden} inductively to~$K(Q^i_0)$, we obtain that $M^i_0\cup P_0$ can be extended to a cycle~$C_0$ that avoids~$z$.
Let $P_1$ be the shortcut edges of the linear forest $(C_0\setminus P_0)\cup M^i_-$, which all belong to~$K(Q^i_1)$ and form a perfect matching on~$K(A_1)$, where $A_1:=V(Q^i_1)\cap V(M^i_-)$.
Applying Theorem~\ref{thm:cycle} inductively to~$K(Q^i_1)$, we obtain that $M^i_1\cup P_1$ can be extended to a cycle~$C_1$.
Observe that $C:=(C_0\setminus P_0)\cup M^i_-\cup (C_1\setminus P_1)$ is a single cycle in~$K(Q_d)$ that extends~$M$ and avoids~$z$, as desired.
Indeed, $C$ is obtained from~$C_1$ by replacing each edge of~$P_1$ by the path between the same end vertices starting and ending with an edge from~$M^i_-$ plus edges of~$C_0\setminus P_0$ in between.

\setlength{\fboxrule}{0.6pt}
\setlength{\fboxsep}{1.5pt}

\vspace{1ex}
\textbf{Case~2:} $|M^i_-|\geq 5$ is odd.
We first give an outline of the construction steps in this case.
Several details are filled in subsequently, and we also later verify that all assumptions needed to apply the various theorems are indeed satisfied.
The points with missing information in the outline are labeled \fbox{1}--\fbox{5}.

The cycle we construct that extends~$M$ uses one additional edge between~$Q^i_0$ and~$Q^i_1$ in addition to the edges~$M^i_-$.
For this purpose we carefully choose a vertex~$u$ in~$Q^i_0$ different from~$z$ such that $uu^i\notin M$ \fbox{1}, and we take $uu^i$ as the edge to be included in the cycle.
Depending on whether and how its end vertices~$u$ and~$u^i$ are covered by~$M$, this creates different conditions in~$Q^i_0$ and~$Q^i_1$ for the induction step.
Specifically, there are three possible cases for~$u$, namely $u \notin V(M)$, $u^M \in V(Q^i_0)$, or $u^M \in V(Q^i_1)$, and similarly three cases for $u^i$; see Figure~\ref{fig:matchingN}.
Because of our assumption that $M$ is \propH{}-maximal, the case $u,u^i\notin V(M)$ cannot occur, as Lemma~\ref{lem:H-violation} would yield that~$|M^i_-|$ is even.
Consequently, for every $u\in V(Q^i_0)$ with $u\neq z$ we have that $u\notin V(M)$ implies that $u^i\in V(M)$.

\tikzset{
  level 1/.style={sibling distance=40mm},
  level 2/.style={sibling distance=10mm},
  vertex/.style = {draw, circle, minimum size=22pt,inner sep=0pt, align=center},
  edge/.style = {line width=0.5mm},
}

\begin{figure}
\begin{tikzpicture}[scale=1]

\node at (0,8.5) {$u \notin V(M)$};
\node at (5,8.5) {$u^M \in V(Q^i_0)$};
\node at (10,8.5) {$u^M \in V(Q^i_1)$};

\node[rotate=90] at (-2,6) {$u^{i} \notin V(M)$};
\node[rotate=90] at (-2,0) {$u^{iM} \in V(Q^i_1)$};
\node[rotate=90] at (-2,-6) {$u^{iM} \in V(Q^i_0)$};

\begin{scope}[shift={(0,5.5)}]
\node at (0,0) {does not occur};
\end{scope}

\begin{scope}[shift={(5,5.5)}]
\draw[dashed] (0,-2) -- (0,2);
\node[vertex,blue] at (-1, 0) (u) {$u$};
\node[vertex] at (1, 0) (ui) {$u^i$};
\node[vertex] at (-1, -2) (um) {$u^M$};
\draw[edge,dotted] (u) -- (ui);
\draw[edge,red] (u) -- (um);
\draw[edge,green] (ui) -- (um);
\node at (-1,2) {$Q^i_0$};
\node at (1,2) {$Q^i_1$};
\node at (0,2.2) {$i$};
\end{scope}

\begin{scope}[shift={(10,5.5)}]
\draw[dashed] (0,-2) -- (0,2);
\node[vertex,blue] at (-1, 0) (u) {$u$};
\node[vertex] at (1, 0) (ui) {$u^i$};
\node[vertex] at (1, -2) (um) {$u^M$};
\draw[edge,dotted] (u) -- (ui);
\draw[edge,red] (u) -- (um);
\draw[edge,green] (ui) -- (um);
\end{scope}

\begin{scope}[shift={(0,0)}]
\draw[dashed] (0,-2) -- (0,2);
\node[vertex] at (-1, 0) (u) {$u$};
\node[vertex,blue] at (1, 0) (ui) {$u^i$};
\node[vertex] at (1, 2) (uim) {$u^{iM}$};
\draw[edge,dotted] (u) -- (ui);
\draw[edge,red] (ui) -- (uim);
\draw[edge,green] (u) -- (uim);
\end{scope}

\begin{scope}[shift={(5,0)}]
\draw[dashed] (0,-2) -- (0,2);
\node[vertex,blue] at (-1, 0) (u) {$u$};
\node[vertex,blue] at (1, 0) (ui) {$u^i$};
\node[vertex] at (1, 2) (uim) {$u^{iM}$};
\node[vertex] at (-1, -2) (um) {$u^M$};
\draw[edge,dotted] (u) -- (ui);
\draw[edge,red] (ui) -- (uim);
\draw[edge,red] (u) -- (um);
\draw[edge,green] (um) -- (uim);
\end{scope}

\begin{scope}[shift={(10,0)}]
\draw[dashed] (0,-2) -- (0,2);
\node[vertex,blue] at (-1, 0) (u) {$u$};
\node[vertex,blue] at (1, 0) (ui) {$u^i$};
\node[vertex] at (1, 2) (uim) {$u^{iM}$};
\node[vertex] at (1, -2) (um) {$u^M$};
\draw[edge,dotted] (u) -- (ui);
\draw[edge,red] (ui) -- (uim);
\draw[edge,red] (u) -- (um);
\draw[edge,green] (uim) to [bend left] (um);
\end{scope}

\begin{scope}[shift={(0,-5.5)}]
\draw[dashed] (0,-2) -- (0,2);
\node[vertex] at (-1, 0) (u) {$u$};
\node[vertex,blue] at (1, 0) (ui) {$u^i$};
\node[vertex] at (-1, 2) (uim) {$u^{iM}$};
\draw[edge,dotted] (u) -- (ui);
\draw[edge,red] (ui) -- (uim);
\draw[edge,green] (u) -- (uim);
\end{scope}

\begin{scope}[shift={(5,-5.5)}]
\draw[dashed] (0,-2) -- (0,2);
\node[vertex,blue] at (-1, 0) (u) {$u$};
\node[vertex,blue] at (1, 0) (ui) {$u^i$};
\node[vertex] at (-1, 2) (uim) {$u^{iM}$};
\node[vertex] at (-1, -2) (um) {$u^M$};
\draw[edge,dotted] (u) -- (ui);
\draw[edge,red] (ui) -- (uim);
\draw[edge,red] (u) -- (um);
\draw[edge,green] (uim) to [bend right] (um);
\end{scope}

\begin{scope}[shift={(10,-5.5)}]
\draw[dashed] (0,-2) -- (0,2);
\node[vertex,blue] at (-1, 0) (u) {$u$};
\node[vertex,blue] at (1, 0) (ui) {$u^i$};
\node[vertex] at (-1, 2) (uim) {$u^{iM}$};
\node[vertex] at (1, -2) (um) {$u^M$};
\draw[edge,dotted] (u) -- (ui);
\draw[edge,red] (ui) -- (uim);
\draw[edge,red] (u) -- (um);
\draw[edge,green] (um) -- (uim);
\end{scope}

\end{tikzpicture}
\caption{Illustration of the definition of the matching~$N$ obtained by modifying~$M$, depending on whether $u \notin V(M)$, $u^M \in V(Q^i_0)$ or $u^M \in V(Q^i_1)$ (three columns) and similarly for $u^i$ (three rows).
The dotted black line is the non-edge $uu^i\notin M$.
The red edges from~$M$ are removed, and the green edge is added to~$N$.
Blue vertices have to be avoided by the cycles~$C_0$ and~$C_1$ that extend~$N^i_0\cup P_0$ and~$N^i_1\cup P_1$, respectively.
}
\label{fig:matchingN}
\end{figure}
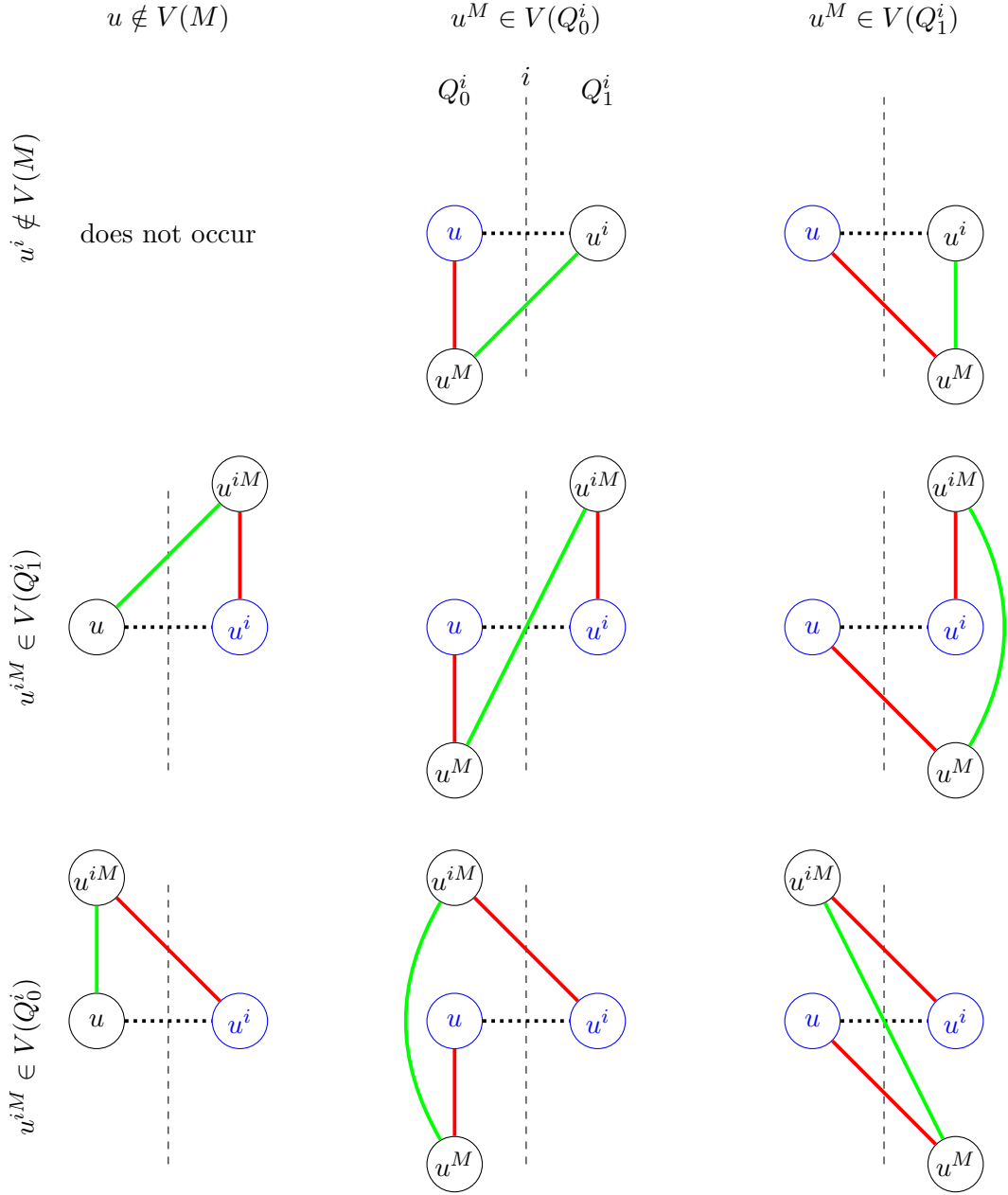

We create a modified matching~$N$ from~$M$ as follows; see Figure~\ref{fig:matchingN}:
If~$u \in V(M)$ we remove the edge~$uu^M$, and if~$u^i \in V(M)$ we remove the edge~$u^iu^{iM}$.
Furthermore, we add the edge~$vw$ where
\begin{equation*}
v:=\begin{cases} u & u\notin V(M), \\ u^M & u\in V(M), \end{cases} \quad\text{and}\quad w:=\begin{cases} u^i & u^i\notin V(M), \\ u^{iM} & u^i\in V(M). \end{cases}
\end{equation*}
We will construct a cycle~$C$ that extends~$N$ and avoids~$z$ and~$u$ if~$u \in V(M)$ as well as~$u^i$ if $u^i \in V(M)$.
Written compactly, the cycle~$C$ avoids~$z$ and~$\{u,u^i\}\setminus\{v,w\}$.
From~$C$, the desired cycle that extends~$M$ and avoids~$z$ can be obtained by straightforward modifications.
Specifically, we remove the edge~$vw$ from~$C$, add the edge~$uu^i$ and the edge $uu^M$ if $u\in V(M)$ as well as $u^iu^{iM}$ if $u^i\in V(M)$.

The construction of~$C$ proceeds as follows:
We carefully choose a perfect matching~$P_0$ on~$K(B_0)$, where $B_0:=V(Q^i_0)\cap V(N^i_-)$ \fbox{2}.
Having chosen~$u$ and~$P_0$, we apply Theorem~\ref{thm:forbidden} inductively to~$K(Q^i_0)$ and obtain that $N^i_0 \cup P_0$ can be extended to a cycle~$C_0$ that avoids~$z$.
If $u \in V(M)$, then we argue that $C_0$ also avoids~$u$ \fbox{3}.
Let $P_1$ be the shortcut edges of the linear forest~$(C_0\setminus P_0)\cup N^i_-$.
If $u^i \notin V(M)$, then we apply Theorem~\ref{thm:cycle} and obtain that $N^i_1 \cup P_1$ can be extended to a cycle~$C_1$.
Otherwise, we apply Theorem~\ref{thm:forbidden} and obtain that $N^i_1\cup P_1$ can be extended to a cycle~$C_1$ that avoids~$u^i$.
As mentioned before, we need to argue that the assumptions of those theorems are met in the subcubes~$Q^i_0$ and~$Q^i_1$ for the matchings $N^i_0\cup P_0$ and~$N^i_1\cup P_1$, and this is where our choices of~$u$ and~$P_0$ become crucial; these verifications are done below under the labels~\fbox{4} and~\fbox{5} separately for the 0- and 1-subcube, respectively.
Observe that $C:=(C_0 \setminus P_0) \cup N^i_- \cup (C_1 \setminus P_1)$ is a single cycle that extends~$N$ and avoids~$z$ and~$u$ if~$u\in V(M)$ as well as~$u^i$ if~$u^i\in V(M)$.

After giving this outline, we now provide the missing details for the points~\fbox{1}--\fbox{5}.
Each of the two steps~\fbox{1} (choosing~$u$) and~\fbox{2} (choosing~$P_0$) splits into two cases.
The argument for~\fbox{3} ($C_0$ avoids $u$) is straightforward and is presented after step~\fbox{1}, as it requires the definition of~$u$.
The verifications~\fbox{4} and~\fbox{5} are done after step~\fbox{1} and again after step~\fbox{2} via some auxiliary claims (in each of the two respective cases), as these two steps are sequential and depend on each other.

\vspace{1ex}
\fbox{1} \textbf{Choosing~$u$.}

Let $X_0:=\{x\in V(Q^i_0)\mid x\neq z\text{ and } x\notin V(M)\}$, i.e., these are vertices in~$V(Q^i_0)$ different from~$z$ that are avoided by~$M$.
Furthermore, let $X_1:=\{x\in V(Q^i_1)\mid x\notin V(M)\}$, i.e., these are vertices in~$V(Q^i_1)$ avoided by~$M$.
As $|M^i_-|$ is odd, we have that $|X_0|$ is even and $|X_1|$ is odd (in particular, $X_1$ is nonempty).

\vspace{0.5ex}
\textbf{Case~a:} $X_0\neq \emptyset$.
We partition~$X_0$ into the sets $X_{01}:=\{x\in X_0\mid x^{iM}\in V(Q^i_1)\}$ and $X_{00}:=\{x\in X_0\mid x^{iM}\in V(Q^i_0)\}$ and we distinguish two cases.

\textit{Case~ai:} If $X_{01}\neq \emptyset$, then we define $u:=x$ for some $x\in X_{01}$.
Note that $N^i_0=M^i_0$, and thus we obtain from~(i) that $N^i_0$ contains no (covered near) half-layers of~$Q^i_0$.

\textit{Case~aii:} If $X_{01}=\emptyset$, then we have $X_0=X_{00}$.
As this set has even cardinality and is nonempty, we have $|X_{00}|\ge 2$, i.e., there are at least two distinct vertices $x,y\in X_0$ with $x^{iM},y^{iM}\in V(Q^i_0)$.
From~(i) we know that $M^i_0$ contains no (covered near) half-layers of~$Q^i_0$.
We argue that (at least) one of $M^i_0\cup\{xx^{iM}\}$ or $M^i_0\cup\{yy^{iM}\}$ contains no (covered near) half-layers of~$Q^i_0$, which implies that we can define $u:=x$ or $u:=y$, respectively.
Suppose for the sake of contradiction that this is false, i.e., both $M^i_0\cup\{xx^{iM}\}$ and $M^i_0\cup\{yy^{iM}\}$ contain a (covered near) half-layer of~$Q^i_0$.
As $M^i_0\cup\{xx^{iM}\}$ contains a (covered near) half-layer, we obtain that $M^i_0$ contains a 2-near half-layer~$L_x$ of~$Q^i_0$, where $x$ is one of the extension vertices and the remaining three extension vertices are covered by~$M$; see Figure~\ref{fig:proof1}.
Symmetrically, as $M^i_0\cup\{yy^{iM}\}$ contains a (covered near) half-layer, we obtain that $M^i_0$ contains a 2-near half-layer~$L_y$ of~$Q^i_0$, where $y$ is one of the extension vertices and the remaining three extension vertices are covered by~$M$.
As all extension vertices of~$L_x$ and~$L_y$ except~$x$ or~$y$, respectively, are covered by~$M$, we obtain that~$L_x\cap L_y=\emptyset$.
On the other hand, by Lemma~\ref{lem:hlayers-number}~(iii), $L_x$ and~$L_y$ have the same direction.
This implies that there is a direction~$j\in[d]\setminus\{i\}$ such that $|M^j_-|\ge |L_x|+|L_y|=2\cdot (2^{d-3}-2)\ge 12$ and $|M^i_-|\le 2\cdot 4=8$.
In fact, as $z,x,y\notin V(M)$ the latter bound can be improved to~$|M^i_-|\leq 5$.
This however contradicts our rules~(1)--(3) for choosing the direction~$i$ (see property~(iii)).

\begin{figure}[h!]
\centerline{
\includegraphics[page=2]{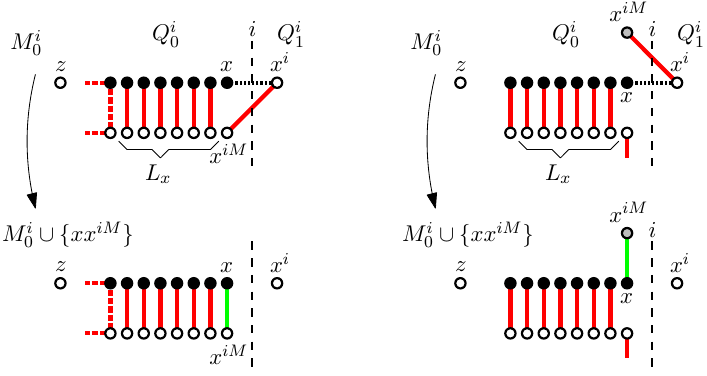}
}
\caption{Illustration of Case~aii in the proof of Theorem~\ref{thm:forbidden}.}
\label{fig:proof1}
\end{figure}

\fbox{4}
{\it Claim 4a: $N^i_0$ contains no (covered near) half-layers of~$Q^i_0$.}
This was argued in the two subcases before.
 
\fbox{5} 
{\it Claim 5a: If $N^i_1$ contains a half-layer of~$Q^i_1$, then there is a vertex on $u^i$'s side of the half-layer that is avoided by~$M$.}
From this claim it follows in particular that $N^i_1$ and~$u^i$ satisfy property~\propH{} in Theorem~\ref{thm:forbidden} for~$Q^i_1$.
To prove the claim, note that we have $N^i_1=M^i_1\setminus\{u^iu^{iM}\}$ (in case~ai) or $N^i_1=M^i_1$ (in case~aii), and therefore $N^i_1$ can only contain a half-layer of~$Q^i_1$ if $M^i_1$ contains a half-layer of~$Q^i_1$.
Let $L$ be such a half-layer of~$Q^i_1$ contained in~$M^i_1$, and let $j\in[d]\setminus\{i\}$ be the direction of~$L$.
By~(ii), this means that $M^i_-$ contains a (covered near) quad-layer of~$Q_d$; see Figure~\ref{fig:proof2}.
Lemma~\ref{lem:hlayers-number}~(i) yields that $M^i_1$ contains no half-layers of~$Q^i_1$ in a direction different from~$j$, and thus $L$ is the only half-layer of~$Q^i_1$ in~$M^i_1$.
If $u^i\in V(L)$, then $N^i_1=M^i_1\setminus\{u^iu^{iM}\}$ contains no half-layers, and there is nothing to show.
If $u^i\notin V(L)$, then the vertex~$u^i$ together with all vertices in~$X_1$ lie in either~$(Q^i_1)^j_0$ or~$(Q^i_1)^j_1$.
As $X_1$ is nonempty, this proves the claim.

\begin{figure}
\centerline{
\includegraphics[page=3]{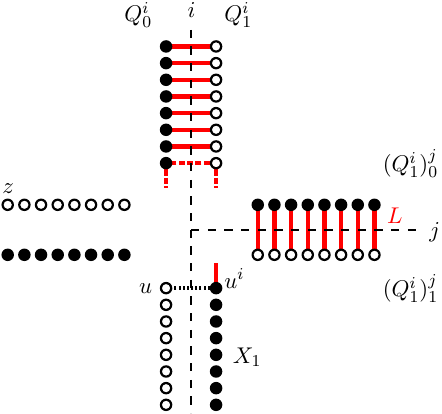}
}
\caption{Illustration of the situation when $M^i_1$ contains a half-layer~$L$ of~$Q^i_1$.}
\label{fig:proof2}
\end{figure}

\vspace{0.5ex}
\textbf{Case~b:} $X_0=\emptyset$.
We define $Y:=\{x\in V(Q^i_0)\mid x \text{ is odd and } xx^i\notin M\}$, i.e., these are odd vertices in~$V(Q^i_0)$ for which the incident edge in direction~$i$ is not in~$M$.
Clearly, we have $z\notin Y$ as $z=\emptyset$ is even.
Furthermore, as $X_0=\emptyset$ and $M$ is assumed to satisfy property~\propH{}, we know that~$Y\neq \emptyset$.

We select~$u\in Y$ according to the following rules applied in this order:
\begin{enumerate}[label=(\arabic*'),leftmargin=8mm, noitemsep, topsep=1pt plus 1pt]
\item If $M^i_1$ contains a near half-layer~$L'$ of~$Q^i_1$, then we select $u\in Y$ as a vertex such that $u^i$ is an (even) end vertex of one of the edges of~$L'$;
\item otherwise, if $M^i_0$ contains a 2-near half-layer~$L''$ of~$Q^i_0$, then we select $u\in Y$ as one of the end vertices of~$L''$;
\item otherwise, we choose $u\in Y$ arbitrarily.
\end{enumerate}

\fbox{4}
{\it Claim 4b: $N^i_0$ contains no (covered near) half-layers of~$Q^i_0$.}
Note that if $N^i_0=M^i_0$ or $N^i_0=M^i_0\setminus\{uu^M\}$, then the claim follows directly from~(i).
It remains to consider the case $N^i_0=(M^i_0\setminus\{uu^M\})\cup\{u^Mu^{iM}\}$, which occurs if and only if $u^M,u^{iM}\in V(Q^i_0)$.
In particular, rule~(1') does not apply, as it would imply~$u^{iM}\in V(Q^i_1)$.
Using~(i) again, it can be ruled out that~$N^i_0$ contains a half-layer of~$Q^i_0$, and it remains to consider covered near half-layers.
If $N^i_0$ contains a covered near half-layer~$L'$, then by~(i) $u^Mu^{iM}$ is one of its edges.
Therefore $M^i_0$ contains the 2-near half-layer~$L'':=L'\setminus\{u^Mu^{iM}\}$; see Figure~\ref{fig:proof3}~(a).
Let $j\in[d]\setminus\{i\}$ be the direction of~$L''$, which is also the direction of the edge~$u^Mu^{iM}$.
By Lemma~\ref{lem:hlayers-number}~(iii), $M^i_0$ contains no 2-near half-layers of~$Q^i_0$ in a direction different from~$j$, so only $L''$ and possibly a second 2-near half-layer in direction~$j$ qualify for application of rule~(2').
However, the edge $uu^M$ removed by rule~(2') also has the direction~$j$.
This is a contradiction, however, as the edges~$uu^M$ and~$u^Mu^{iM}$ have different directions (as $u\neq u^{iM}$), proving the claim.

\begin{figure}[h!]
\centerline{
\includegraphics[page=4]{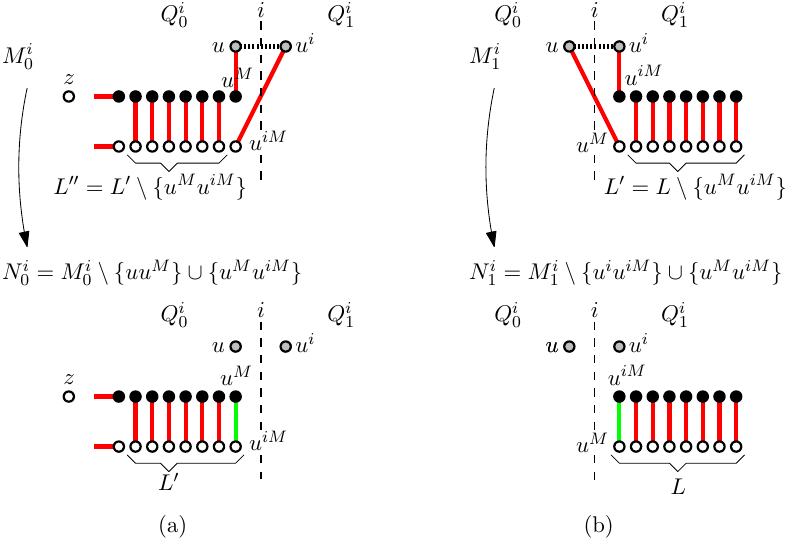}
}
\caption{Illustration of Case~b in the definition of~$u$ in the proof of Theorem~\ref{thm:forbidden}.}
\label{fig:proof3}
\end{figure}

\fbox{5}
{\it Claim 5b: If $u^i\in V(M)$, then $N^i_1$ contains no half-layers of~$Q^i_1$.}
It follows in particular that $N^i_1$ and~$u^i$ satisfy property~\propH{} in Theorem~\ref{thm:forbidden} for~$Q^i_1$ (recall from the proof outline that this property need not be checked if $u^i\notin V(M)$).
To prove the claim, note that if $M^i_1$ contains a half-layer~$L$ of~$Q^i_1$, then by~(ii) and Lemma~\ref{lem:hlayers-number}~(i), $L$ is the only half-layer in~$M^i_1$ and all near half-layers must be contained in~$L$.
By rule~(1'), the edge~$u^iu^{iM}$ is an edge from~$L$, so in both of the possible cases $N^i_1=M^i_1\setminus\{u^iu^{iM}\}$ or $N^i_1=(M^i_1\setminus\{u^iu^{iM}\})\cup\{u^Mu^{iM}\}$ the resulting set~$N^i_1$ contains no half-layers of~$Q^i_1$, and we are done.
On the other hand, if $M^i_1$ contains no half-layers, then $N^i_1$ may only contain a half-layer~$L$ if $N^i_1=(M^i_1\setminus\{u^iu^{iM}\})\cup\{u^Mu^{iM}\}$, which occurs if and only if $u^M,u^{iM}\in V(Q^i_1)$, and then $L':=L\setminus\{u^Mu^{iM}\}$ is a near half-layer in~$M^i_1$; see Figure~\ref{fig:proof3}~(b).
Let $j\in[d]\setminus\{i\}$ be the direction of~$L'$, which is also the direction of the edge~$u^Mu^{iM}$.
By Lemma~\ref{lem:hlayers-number}~(i), $M^i_1$ contains no near half-layers of~$Q^i_1$ in a direction different from~$j$, so only $L'$ and possibly a second near half-layer in direction~$j$ qualify for application of rule~(1').
However, the edge $u^iu^{iM}$ removed by rule~(1') also has the direction~$j$.
This is a contradiction, however, as the edges~$u^iu^{iM}$ and~$u^Mu^{iM}$ have different directions (as $u^i\neq u^M$), proving the claim.

\vspace{1ex}
\fbox{3} {\bf $C_0$ avoids $u$.}
Note that the cycle~$C_0$ obtained from applying Theorem~\ref{thm:forbidden} to extend~$N^i_0\cup P_0$ avoids the vertex~$z$ by construction.
As $N^i_0\cup P_0$ avoids only one other vertex in~$Q^i_0$ apart from~$z$, namely the vertex~$u$ (recall that $X_0=\emptyset$), whose parity is opposite to that of~$z$,  Lemma~\ref{lem:avoid2} implies that~$C_0$ must also avoid the vertex~$u$.

\vspace{1ex}
\fbox{2} \textbf{Choosing~$P_0$.}

From~(iii) and the fact that~$|M^i_-|$ is odd we obtain that~$|M^i_-|\ge 5$, and consequently~$|N^i_-|$ is even and $|N^i_-|\ge 4$.
This assumption is crucial when we will apply Lemma~\ref{lem:avoid-layer} to~$Q^i_0$.
We distinguish the cases~$u^i\notin V(M)$ and~$u^i\in V(M)$.

\vspace{0.5ex}
\textbf{Case~a:} $u^i\notin V(M)$.

\fbox{4}+\fbox{5} We argued before (Claims~4a and 4b) that $N^i_0$ contains no (covered near) half-layers of~$Q^i_0$.
We apply Lemma~\ref{lem:avoid-layer} to obtain a perfect matching~$P_0$ on~$K(B_0)$ so that~$N^i_0\cup P_0$ contains no half-layers of~$Q^i_0$.
It follows that Theorem~\ref{thm:forbidden} applies inductively to~$K(Q^i_0)$ to obtain a cycle~$C_0$ that extends~$N^i_0\cup P_0$ and avoids~$z$.
Furthermore, Theorem~\ref{thm:cycle} applies inductively to~$K(Q^i_1)$, which does not require any assumptions on~$N^i_1\cup P_1$, to obtain a cycle~$C_1$ that extends~$N^i_1\cup P_1$.

\vspace{0.5ex}
\textbf{Case~b:} $u^i\in V(M)$.

Let $I$ be the set of all directions $j\in [d] \setminus\{i\}$ for which there exists a perfect matching~$P'$ of~$K(B_1)$, where $B_1:=V(Q^i_1)\cap V(N^i_-)$ such that~$N^i_1 \cup P'$ and~$u^i$ violate property~\propH{}, i.e., $N^i_1\cup P'$ contains a half-layer in direction~$j$ and all vertices on $u^i$'s side of the half-layer other than~$u^i$ itself are covered.
If $j\in I$, then as $u^i$ is not covered by~$N$, there is exactly one half-layer~$L$ of~$Q^i_1$ in direction~$j$ for which this condition holds.
Furthermore, at least one edge of~$L$ must come from the matching~$P'$ and not from~$N^i_1$.
This follows from our earlier argument (Claims~5a and 5b) that if $N^i_1$ contains a half-layer of~$Q^i_1$, then there is a vertex on~$u^i$'s side of the half-layer that is avoided by~$M$ (because this forces $I=\emptyset$).

We distinguish three cases depending on the size of~$I$.

\textit{Case~bi:} $|I|=0$.

\fbox{4}+\fbox{5} We argued before (Claims~4a and 4b) that $N^i_0$ contains no (covered near) half-layers of~$Q^i_0$.
We apply Lemma~\ref{lem:avoid-layer} to obtain a perfect matching~$P_0$ on~$K(B_0)$ so that~$N^i_0\cup P_0$ contains no half-layers of~$Q^i_0$.
It follows that Theorem~\ref{thm:forbidden} applies inductively to~$K(Q^i_0)$ to obtain a cycle~$C_0$ that extends~$N^i_0\cup P_0$ and avoids~$z$.
The definition of~$I$ ensures that Theorem~\ref{thm:forbidden} applies inductively to~$K(Q^i_1)$ (regardless of~$C_0$) to obtain a cycle~$C_1$ that extends~$N^i_1\cup P_1$ and avoids~$u^i$.

\textit{Case~bii:} $|I|=1$.

We distinguish the subcases~$|N^i_-|\ge 6$ and~$|N^i_-|=4$.

\textit{Case~bii':} $|N^i_-|\ge 6$.

Let $xx^j\notin N^i_1$ be an edge of the unique half-layer of~$Q^i_1$ in direction~$j$ with~$I=\{j\}$ which is contained in~$N^i_1\cup P'$ for some perfect matching~$P'$ of~$K(B_1)$.
The key idea that will prevent~$xx^j$ to appear as a shortcut edge of the linear forest~$(C_0\setminus P_0)\cup N^i_-$ (regardless of~$C_0$) is to include the edge~$x^Nx^{jN}$ in~$P_0$.
This ensures that~$xx^j\notin P_1$ and therefore~$N^i_1\cup P_1$ contains no half-layers of~$Q^i_1$.
This idea was first presented in~\cite{MR3830138} in the proof of their Theorem~3.

\fbox{4}+\fbox{5} We argued before (Claims~4a and 4b) that~$N^i_0$ contains no (covered near) half-layers of~$Q^i_0$.
Consequently, $N^i_0\cup\{x^Nx^{jN}\}$ contains no half-layers of~$Q^i_0$.
Thus, we apply Lemma~\ref{lem:avoid-layer} to $N^i_0\cup\{x^Nx^{jN}\}$ to obtain a perfect matching~$P_0$ of~$K(B_0)$ that includes the edge~$x^Nx^{jN}$ such that $N^i_0\cup P_0$ contains no half-layers of~$Q^i_0$.
It follows that Theorem~\ref{thm:forbidden} applies inductively to~$K(Q^i_0)$ to obtain a cycle~$C_0$ that extends~$N^i_0\cup P_0$ and avoids~$z$.
As $x^Nx^{jN}\in P_0$, we have $xx^j\notin P_1$ (regardless of~$C_0$), and so Theorem~\ref{thm:forbidden} applies inductively to~$K(Q^i_1)$ to obtain a cycle~$C_1$ that extends~$N^i_1\cup P_1$ and avoids~$u^i$.

\textit{Case~bii'':} $|N^i_-|=4$.
In this case we have~$|M^i_-|=5$.
From~(iii) we obtain that rule~(3) was applied to choose the direction~$i$, and therefore~$|M^j_-|\le |M^i_-|=5$ for all $j\in[d]\setminus\{i\}$.
As $P_1$ has only two edges and a half-layer of~$Q^i_1$ has~$2^{d-3}\ge 8$ edges, we observe that a half-layer of~$Q^i_1$ can be present in~$N^i_1\cup P_1$ only if $d=6$ and $N^i_1$ contains a 2-near half-layer~$L''$ in direction~$j$ with~$I=\{j\}$ which has 6~edges, one of them being the edge~$u^Mu^{iM}$, and the four extension vertices of~$L''$ are the end vertices of the four edges in~$N^i_-$; see Figure~\ref{fig:proof4}.
Furthermore, in this case $N^i_0$ contains no 2-near half-layers of~$Q^i_0$ (as this would require 6 edges of~$Q^i_0$ in the same direction).
We let $xx^j,yy^j\notin N^i_1$ be the two edges such that~$L''\cup\{xx^j,yy^j\}$ is a half-layer of~$Q^i_1$, and we define $P_0:=\{x^Nx^{jN},y^Ny^{jN}\}$.

\fbox{4}+\fbox{5} As $N^i_0$ contains no 2-near half-layers of~$Q^i_0$, we have that $N^i_0\cup P_0$ contains no half-layers of~$Q^i_0$, so Theorem~\ref{thm:forbidden} applies inductively to~$K(Q^i_0)$ to obtain a cycle~$C_0$ that extends~$N^i_0\cup P_0$ and avoids~$z$.
As $x^Nx^{jN}\in P_0$, we have $xx^j\notin P_1$ (regardless of~$C_0$), so Theorem~\ref{thm:forbidden} applies inductively to~$K(Q^i_1)$ to obtain a cycle~$C_1$ that extends~$N^i_1\cup P_1$ and avoids~$u^i$.

\begin{figure}[h!]
\centerline{
\includegraphics[page=5]{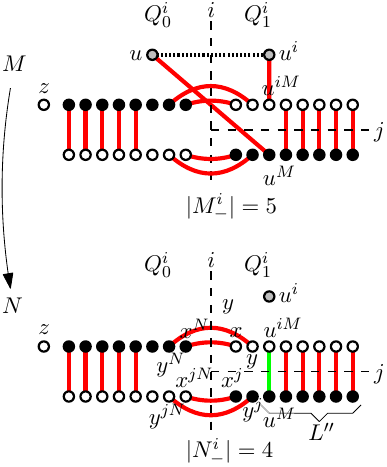}
}
\caption{Illustration of Case~bii'' in the proof of Theorem~\ref{thm:forbidden}.}
\label{fig:proof4}
\end{figure}

\textit{Case~biii:} $|I|\ge 2$.

For any~$j\in I$, let $L_j$ be the half-layer of~$Q^i_1$ in direction~$j$ that is contained in~$N^i_1\cup P'$ for some perfect matching~$P'$ on~$K(B_1)$ (recall that $L_j$ is unique as~$u^i\notin L_j$).
For any two half-layers~$L_j,L_k$, $j,k\in I$, $j\neq k$, Lemma~\ref{lem:hlayers-structure}~(i) yields that the union~$L_j\cup L_k$ is a collection of paths of length~2.
Consequently, if an edge from~$L_j$ did belong to~$N^i_1$, then its neighboring edge from~$L_k$ would not belong to~$N^i_1$, but at the same time their common end vertex would not be in~$B_1$, contradicting the fact that~$k\in I$.
It follows that~$V(L_j)\subseteq B_1$ for all $j\in I$, in other words, none of the edges in any of the half-layers~$L_j$ belongs to~$N^i_1$.
By Lemma~\ref{lem:hlayers-structure}~(ii), the union~$\bigcup_{j\in I} L_j$ avoids exactly~$2^{d-1-|I|}$ vertices of~$Q^i_1$ and therefore $|N^i_-|\ge 2^{d-1}-2^{d-1-|I|}$, which implies $|N^i_0|\le (2^{d-1}-|N^i_-|)/2\le 2^{d-2-|I|}$.
For every `dangerous' direction~$j\in I$ we will pick an edge~$xx^j\in L_j$, and we add the edge~$x^Nx^{jN}$ (which lies in~$Q^i_0$ by the arguments from before) to the matching~$P_0$, which ensures that $xx^j\notin P_1$.
Specifically, we define~$v:=u^i$ and we consider the edge~$e(j):=v^{p(j)}(v^{p(j)})^j\in L_j$ for every $j\in I$, where $p(j):=j-1\pmod d$.
Note that the edges~$e(j)$, $j\in I$, are independent, i.e., no two of them share an end vertex, so for every~$j\in I$ we let $xx^j:=e(j)$, and we add the edge~$x^N x^{jN}$ to~$P_0$.
As $|N^i_0|+|I|\le 2^{d-2-|I|}+|I|<2^{d-3}$, this does not create half-layers in~$N^i_0\cup P_0$.
The matching~$P_0$ on~$K(B_0)$ is completed by applying Lemma~\ref{lem:avoid-layer}, using that $|N^i_-|-2|I|\ge 2^{d-1}-2^{d-1-|I|}-2|I|\ge 4$.

\fbox{4}+\fbox{5} The argument then continues analogously as in case~bii' before.

This completes the proof.
\end{proof}

\subsection{Reverse implication (induction basis $d=5$)} \label{sec:basis}

It remains to settle the base case~$d=5$ in Theorem~\ref{thm:forbidden}, which we do with computer assistance.
In this section we describe the underlying theoretical considerations.
Our verification program in C++ is available for download from Gitlab~\cite{gitlab}, and it spans approximately 1400 lines of code.
The program is compiled using GCC~12.2.0 on Debian 12, and it runs as a single thread.
The reported running times are obtained on an AMD Ryzen~9 7900X3D, 4.4~GHz, 64~GB RAM desktop computer.

Note that $K(Q_5)$ has $31\cdot 29\cdot 27\cdots 3>10^{17}$ many perfect matchings, considerably more than what could be handled by a naive exhaustive enumeration approach.
This explains the following theoretical considerations necessary for being able to perform these verifications in reasonable computing time.

Let~$M$ be a matching of~$K(Q_5)$ that avoids~$z$ and satisfies property~\propH{}.
We assume w.l.o.g.\ that $M$ is \propH{}-maximal.
We distinguish two cases, namely if $M$ contains a $z$-dangerous (near) half-layer or not.

\subsubsection{$M$ contains a $z$-dangerous (near) half-layer}

If $M$ contains a $z$-dangerous (near) half-layer, then the computer verification proceeds as follows:
We generate all matchings of~$K(Q_5)$ that contain a $z$-dangerous near half-layer in some direction~$i$ (w.l.o.g.\ we can assume $i=1$), and for each of them that satisfies property~\propH{}, we verify that it can be extended to a Hamilton cycle that avoids~$z$.
We reduce the number of test cases by considering the group of automorphisms generated by permutations of all directions except~$i$, and we only consider matchings that are non-isomorphic under this group action.
The running time for completing the verifications in this case is about 3~minutes.

\subsubsection{$M$ does not contain $z$-dangerous (near) half-layers}
\label{sec:no-nhlayer}

We now assume that $M$ does not contain $z$-dangerous (near) half-layers.
By Lemma~\ref{lem:H-violation}, $M$ is \defi{maximal} in the sense that it covers at least one end vertex of every edge of~$Q_5\setminus z$.
This is helpful as testing maximality of a matching is easier than testing property~\propH{}.

We select a direction~$i\in[5]$ according to the following rules applied in this order:
\begin{enumerate}[leftmargin=8mm, noitemsep, topsep=1pt plus 1pt]
\item If $M$ contains a $z$-dangerous (near) quad-layer, then we choose~$i$ according to Lemma~\ref{lem:Q5-quad-cut} (note that Lemma~\ref{lem:hlayers-number}~(ii) does \emph{not} apply for $d=5$);
\item otherwise, we choose a direction $i \in [5]$ that maximizes the quantity~$|M^i_-|$.
\end{enumerate}
By Lemmas~\ref{lem:Q5-quad-cut} and~\ref{lem:H-maximal}, these rules guarantee that $M^i_0$ contains no (near) half-layers of~$Q^i_0$ and that $|M^i_-|\ge 3$.
We now distinguish two cases depending on the parity of~$|M^i_-|$.

\textbf{Case~1:} $|M^i_-|\geq 4$ is even.

Our program constructs the following sets of matchings:
\begin{itemize}[leftmargin=6mm, noitemsep, topsep=1pt plus 1pt]
\item $\cM$: The set of maximal matchings of~$K(Q^i_0\setminus z)$ that do not contain (near) half-layers of~$Q^i_0$.
\item $\cM'$: The matchings in~$\cM$ that cannot be extended to a cycle of~$K(Q^i_0\setminus z)$.
\item $\cM_2$: The set of pairs~$(M,A)$ where $M$ is obtained from some matching in~$\cM'$ by removing two of its edges and $A$ are the four end vertices of the removed edges.
\item $\cM_2'$: The subset of pairs $(M,A)$ from~$\cM_2$ such that $M$ cannot be extended to a linear forest of $K(Q^i_0\setminus z)$ with terminals exactly in~$A$.
\end{itemize}

Our program reports that the set~$\cM_2'$ is empty, and we now argue that this implies that $M$ can be extended to a cycle in~$K(Q_5)$ that avoids~$z$.

Indeed, we define $A_0:=V(Q^i_0)\cap V(M^i_-)$ and apply Lemma~\ref{lem:avoid-layer} to obtain a perfect matching~$P_0$ on~$K(A_0)$ such that $N:=M^i_0 \cup P_0$ contains no (near) half-layers of~$Q^i_0$.
As $M$ and $N$ have the same set of uncovered vertices in~$Q^i_0$, $N$ is a maximal matching of~$K(Q^i_0\setminus z)$, so $N \in \cM$.
If $N\notin \cM'$, then $N$ can be extended to a cycle of~$K(Q^i_0\setminus z)$, and we can apply Theorem~\ref{thm:cycle} to~$Q^i_1$ and combine the two cycles to a cycle of~$K(Q_5)$ that extends~$M$ and avoids~$z$.

It remains to argue about the case $N\in \cM'$.
If $|M^i_-|=4$, then we have $(M^i_0,A_0) \in \cM_2$.
As $\cM_2'=\emptyset$, $M^i_0$ can be extended to a linear forest of~$K(Q^i_0\setminus z)$ with terminals exactly in~$A_0$, and we can apply Theorem~\ref{thm:cycle} to~$Q^i_1$ as before.
If $|M^i_-| \ge 6$, then let $P_0'$ be the matching obtained from~$P_0$ by removing two of its edges.
From the definition, we have $(M^i_0 \cup P_0', A_0 \setminus V(P_0')) \in \cM_2$, and so the argument continues similarly to before.

The computations for this case take about 1~second.
This is achieved again by removing isomorphic matchings early on during the computation.

\textbf{Case~2:} $|M^i_-|\geq 3$ is odd.

\begin{figure}[h!]
\centerline{
\includegraphics[page=7]{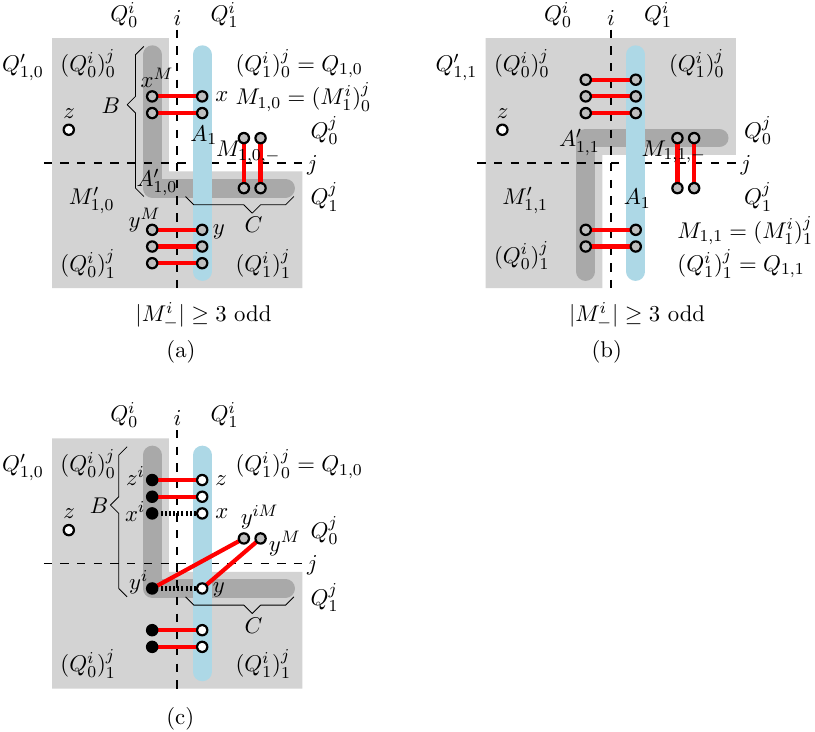}
}
\caption{Illustration of definitions used in Case~2 of Section~\ref{sec:no-nhlayer}.}
\label{fig:proof5}
\end{figure}

The following arguments are illustrated in Figure~\ref{fig:proof5}.
We define $A_1:=V(Q^i_1)\cap V(M^i_-)$.
For some direction~$j\in[5]$, $j\neq i$, we define $Q_{1,0}:=(Q^i_1)^j_0$ and $Q_{1,1}:=(Q^i_1)^j_1$.
We also define $Q_{1,0}':=Q_5\setminus Q_{1,0}$ and $Q_{1,1}':=Q_5\setminus Q_{1,1}$.
Let $M_{1,0}:=(M^i_1)^j_0$ and~$M_{1,1}:=(M^i_1)^j_1$ be the set of edges of~$M$ having both end vertices in~$Q_{1,0}$ and~$Q_{1,1}$, respectively.
Similarly, let $M_{1,0}'$ and $M_{1,1}'$ be the set of edges of~$M$ having both end vertices in~$Q_{1,0}'$ and~$Q_{1,1}'$, respectively.
Furthermore, let $M_{1,0,-}$ and $M_{1,1,-}$ be the edges of $M$ having exactly one end vertex in $Q_{1,0}$ and $Q_{1,1}$, respectively, and let $A_{1,0}'$ and~$A_{1,1}'$ be the vertices of~$Q_{1,0}'$ and~$Q_{1,1}'$ covered by edges of~$M_{1,0,-}$ and $M_{1,1,-}$.
We write~$L^j_{0i}$ for the half-layer in~$Q^j_0$ in direction~$i$ for which the end vertices whose $i$th bit equals~0 are odd.
This is precisely the $z$-dangerous half-layer of~$Q^j_0$ in direction~$i$.
Similarly, we write~$L^j_{1i}$ for the half-layer in~$Q^j_1$ in direction~$i$ for which the end vertices whose $i$th bit equals~0 are odd.
We define~$\cL:=\{L^j_{0i},L^j_{1i}\mid j\in[d]\setminus\{i\}\}$.

We aim to choose the direction~$j\in [5]$, $j\neq i$, so that one of the following two conditions is satisfied:
\begin{enumerate}[leftmargin=6mm, noitemsep, topsep=1pt plus 1pt,label=(\alph*)]
\item $|A_{1,0}'|\geq 2$ is even and there exists a perfect matching~$P_{1,0}'$ on~$K(A_{1,0}')$ such that $M_{1,0}' \cup P_{1,0}'$ does not contain half-layers of~$Q^i_0$ nor the half-layer~$L^j_{1i}$ of~$Q^j_1$ (Figure~\ref{fig:proof5}~(a)), or
\item $|A_{1,1}'|\geq 2$ is even and there exists a perfect matching~$P_{1,1}'$ on~$K(A_{1,1}')$ such that $M_{1,1}' \cup P_{1,1}'$ does not contain half-layers of~$Q^i_0$ nor the half-layer~$L^j_{0i}$ of~$Q^j_0$ (Figure~\ref{fig:proof5}~(b)).
\end{enumerate}

Note that if a matching of~$K(Q_{1,0}')$ covers all vertices of~$Q^i_0$ except~$z$ and contains the half-layer~$L^j_{1i}$ of~$Q^j_1$, then it cannot be extended to a cycle.
Similarly, if a matching of~$K(Q_{1,1}')$ covers all vertices of~$Q^i_0$ except~$z$ and contains the half-layer~$L^j_{0i}$ of~$Q^j_0$, then it cannot be extended to a cycle.
This observation follows along the same lines as the proof of the necessity of property~\propH{} given in Section~\ref{sec:forward}.
We thus achieve a considerable reduction in the size of the set~$\cM'$ defined below compared to what it would be if the conditions involving~$L^j_{1i}$ and~$L^j_{0i}$ would not be present.

Specifically, for choosing~$j$ we distinguish two cases.

We first consider the case that~$M^i_-$ contains no quad-layers of~$Q_5$ from the set~$\cL$.
We choose two vertices $x,y \in A_1$ and if possible choose them so that $xx^i, yy^i \in M$.
We then choose the direction~$j$ so that it separates~$x$ and~$y$, i.e., one is contained in~$Q^j_0$ and the other one in~$Q^j_1$, implying that $|A_{1,0}'|>0$ and~$|A_{1,1}'|>0$.
Since $|A_1|=|M^i_-|$ is odd, $|A_{1,0}'|$ and $|A_{1,1}'|$ have different parity, and so one of them is even and at least~2.
By symmetry, we assume w.l.o.g.\ that $|A_{1,0}'|\geq 2$ is even and that~$x\in Q_{1,0}$; see Figure~\ref{fig:proof5}~(a).
We define~$B:=A_{1,0}'\cap V(Q^i_0)$ and~$C:=A_{1,0}'\setminus B$.
Note that we have~$x^M\in B$.
The matching~$P_{1,0}'$ on~$K(A_{1,0}')$ is obtained as follows.
If $|B|\geq 4$, then we apply Lemma~\ref{lem:avoid-layer} to match all vertices in~$B$, except the vertex~$x^M$ if $B$ is odd, such that no near half-layers of~$Q^i_0$ are created.
If $|B|\in\{2,3\}$, then we match one pair of vertices from~$B$, except the vertex~$x^M$ if $|B|=3$, which does not create half-layers of~$Q^i_0$.
If $|B|=|\{x^M\}|=1$ or the vertex~$x^M$ from~$B$ is still unmatched from the previous cases, we match it arbitrarily to one of the vertices in~$C$, and the remaining vertices in~$C$ we match arbitrarily among each other.
Note that adding an edge from~$x^M\in B$ to some vertex~$c\in C$ cannot create the half-layer~$L^j_{1i}$, as this would imply~$(x^M)^i\in C$ and therefore $(x^M)^i\neq x$ and $xx^i\notin M$, but since 3 edges of~$L^j_{1i}$ are in~$M^i_-$, this contradicts the choice of~$x$.

It remains to consider the case that~$M^i_-$ contains a quad-layer~$L\in\cL$ of~$Q_5$, which has $2^{d-3}=4$~edges.
As $M$ does not contain $z$-dangerous (near) half-layers of~$Q_5$, there are at least two edges~$xx^i$ and $yy^i$ with $x,y\in V(Q^i_1)$ of the $z$-dangerous half-layer in direction~$i$ that are not contained in~$M^i_-$; see Figure~\ref{fig:proof5}~(c).
We choose the direction~$j$ so that it separates~$x$ and~$y$.
As before, we obtain that one of~$|A_{1,0}'|$ and $|A_{1,1}'|$ is even and at least~2, and by symmetry we may assume w.l.o.g.\ that~$|A_{1,0}'|\geq 2$ is even and that~$x\in Q_{1,0}$.
As $yy^i\in L^j_{1i}$ but $yy^i\notin M^i_-$, the matching~$M_{1,0}'$ does not contain the half-layer~$L^j_{1i}$ of~$Q^j_1$.
Furthermore, as $|L|=4$, at least one of the edges~$zz^i\in L$ satisfies~$z\in Q_{1,0}$.
We define~$B:=A_{1,0}'\cap V(Q^i_0)$ and~$C:=A_{1,0}'\setminus B$.
Note that we have~$z^i\in B$.
The matching~$P_{1,0}'$ on~$K(A_{1,0}')$ is obtained is follows.
If $y^i\in B$ and~$y\in C$, which happens if~$y^{iM},y^M\in V(Q_{1,0})$, then we first match~$z^i$ to~$y$ and define~$B:=B\setminus\{y^i\}$ and~$C:=C\setminus\{y\}$.
If $|B|\geq 4$, then we apply Lemma~\ref{lem:avoid-layer} to match all vertices in~$B$, except one if $B$ is odd, such that no near half-layers of~$Q^i_0$ are created.
If $|B|\in\{2,3\}$, then we match one pair of vertices from~$B$, except one vertex if $|B|=3$, which does not create half-layers of~$Q^i_0$.
If $|B|=1$ or one vertex from~$B$ is still unmatched from the previous cases, we match it arbitrarily to one of the vertices in~$C$, and the remaining vertices in~$C$ we match arbitrarily among each other.
Note that $yy^i\notin P_{1,0}'$, and therefore $M_{1,0}'\cup P_{1,0}'$ does not contain the half-layer~$L^j_{1i}$.

We proceed to explain the further arguments assuming that there is a matching~$P_{1,0}'$ on~$K(A_{1,0}')$ that satisfies condition~(a) stated above.
The arguments in case~(b) are analogous, but not completely symmetric (due to~$z$), so our algorithm has to compute both cases separately.
As mentioned before, we can assume that~$M$ is maximal in the sense that it covers at least one end vertex of every edge of~$Q_5\setminus z$, and therefore $M_{1,0}' \cup P_{1,0}'$ is maximal in~$Q_{1,0}'\setminus z$.

Our program constructs the following sets of matchings:
\begin{itemize}[leftmargin=6mm, noitemsep, topsep=1pt plus 1pt]
\item $\cM$: The set of maximal matchings of~$K(Q_{1,0}'\setminus z)$ with an odd number of covered vertices in~$Q^i_0$ and no half-layers of~$Q^i_0$ nor the half-layer~$L^j_{1i}$ of~$Q^j_1$.
\item $\cM'$: The matchings in~$\cM$ that cannot be extended to a cycle of~$K(Q_{1,0}'\setminus z)$.
\item $\cM_1$: The set of pairs~$(M,A)$ where $M$ is obtained from some matching in~$\cM'$ by removing one of its edges and $A$ are the two end vertices of the removed edge.
\end{itemize}
We also compute the following sets in a loop for $k=1,2,3,4$.
\begin{itemize}[leftmargin=6mm, noitemsep, topsep=1pt plus 1pt]
\item $\cM_k'$: The subset of pairs~$(M,A)$ from~$\cM_k$ such that~$M$ cannot be extended to a linear forest of~$K(Q_{1,0}'\setminus z)$ with terminals exactly in~$A$.
\item $\ol{\cM}_k$: The set of matchings $M$ of~$K(Q_5)$ for which there is a pair~$(M',A')$ in~$\cM_k'$ such that $M$ induced on~$Q_{1,0}'$ equals $M'$ and the set of end vertices of edges in~$M$ that leave the set~$Q_{1,0}'$ is~$A'$.
\item $\ol{\cM}'_k$: The matchings in~$\ol{\cM}_k$ that cannot be extended to a cycle of~$K(Q_5\setminus z)$.
\item $\cM_{k+1}$: The set of pairs $(M,A)$ obtained from some pair $(M',A')$ in~$\cM_k'$ by removing one of the edges from~$M'$ and adding the two end vertices of the removed edge to the set~$A$.
\end{itemize}
We note that $\cM_1=\cM_1'$.
Our program reports that the sets $\ol{\cM}'_k$ for $k=1,2,3,4$ are all empty.
This proves that~$M$ can be extended to a cycle in~$K(Q_5)$ that avoids~$z$, as argued in the following.

Recall the assumption from case~(a) that $|A_{1,0}'|\geq 2$ is even and there exists a perfect matching~$P_{1,0}'$ on~$K(A_{1,0}')$ such that $M_{1,0}' \cup P_{1,0}'$ does not contain half-layers of~$Q^i_0$ nor the half-layer~$L^j_{1i}$ of~$Q^j_1$.
This means that $M_{1,0}' \cup P_{1,0}' \in \cM$.

If $M_{1,0}' \cup P_{1,0}' \notin \cM'$, then $M_{1,0}' \cup P_{1,0}'$ can be extended to a cycle $C'$ of~$K(Q_{1,0}'\setminus z)$.
Let $P_{1,0}$ be the shortcut edges of the linear forest $(C' \setminus P_{1,0}') \cup M_{1,0,-}$ in~$Q_{1,0}$.
By Theorem~\ref{thm:cycle}, $P_{1,0} \cup M_{1,0}$ can be extended to a cycle $D$ of~$K(Q_{1,0})$.
Therefore, $(C' \setminus P_{1,0}') \cup M_{1,0,-}\cup (D \setminus P_{1,0})$ is a cycle in~$K(Q_5)$ that extends~$M$ and avoids~$z$.

If $M_{1,0}' \cup P_{1,0}' \in \cM'$, then let $N_1$ be a matching that contains all edges of $M_{1,0}' \cup P_{1,0}'$ except one edge of~$P_{1,0}'$ and let $A_1$ be the end vertices of the removed edge.
By definition, we have $(N_1, A_1) \in \cM_1$.
If $(N_1, A_1) \notin \cM_1'$, then $N_1$ be extended to a linear forest of~$K(Q_{1,0}'\setminus z)$ with terminals exactly in~$A_1$.
Similarly as in the previous paragraph, we then construct shortcut edges in~$Q_{1,0}$, use Theorem~\ref{thm:cycle} to obtain a cycle in~$K(Q_{1,0})$ which we combine with the linear forest to obtain a cycle in~$K(Q_5)$ that extends~$M$ and avoids~$z$.

If $(N_1, A_1) \in \cM_1'$ and $|M_{1,0,-}| = 2$, then we have $M \in \ol{\cM}_1$ by definition.
Since $\ol{\cM}'_1$ is empty, $M$ can be extended to a cycle of~$K(Q_5\setminus z)$.

If $(N_1, A_1) \in \cM_1'$ and $|M_{1,0,-}| \ge 4$, then let $N_2$ be a matching that contains all edges of~$N_1$ except another edge of~$P_{1,0}'$ and let $A_2$ be the end vertices of both edges of~$P_{1,0}'$ excluded in~$N_2$.
By definition, we have $(N_2, A_2) \in \cM_2$.
If $(N_2, A_2) \notin \cM_2'$ or $|M_{1,0,-}| = 4$, then we proceed similarly to before.
Otherwise, we construct and process $N_3$ and $A_3$ analogously.
If $(N_3, A_3) \notin \cM_3'$ or $|M_{1,0,-}| = 6$, then we proceed similarly to before.
Otherwise, we construct and process $N_4$ and $A_4$ analogously.
Since $Q_{1,0}$ has 8 vertices, this is the last possible case, i.e., we have $(N_4, A_4) \notin \cM_4'$ or $|M_{1,0,-}| = 8$.

The running time for this part of our program is approximately 10~hours for each of the two cases~(a) and~(b).

This completes the proof of Theorem~\ref{thm:forbidden} in the case~$d=5$.

\subsection{Implementation details}
\label{sec:program}

We now give some details about how the program described in the previous section generates matchings and tests extendability to a cycle or linear forest.

The program represents a matching or a partial extension of a matching to a cycle or linear forest as an array~$M$ of length $2^d$; i.e. one value for every vertex.
For any vertex~$u\in V(Q_d)$, the possible values of~$M[u]$ are the following:
\begin{description}[font=\normalfont,leftmargin=10mm, noitemsep, topsep=1pt plus 1pt]
\item[Vertex $v\in V(Q_d)$] The matching~$M$ has an edge $uv$. Clearly, it is required that $u \neq v$ and $M[v]=u$.
\item[$\forbidden$] The vertex~$u$ is avoided by~$M$ and must be avoided by an extending cycle or linear forest.
This allows us to mask out various sets of vertices of~$Q_d$ that are irrelevant for a particular test case.
\item[$\uncovered$] The vertex~$u$ is avoided by~$M$ and can be a part of an extending cycle or linear forest.
\item[$\terminal$] The vertex~$u$ must be a terminal vertex of a linear forest.
Clearly, the total number of terminals must be even.
\item[$\match$] A special intermediate label used for generating matchings, as explained below.
\end{description}

\subsubsection{Generating matchings}

We now explain how our program generates different matchings using the data structure~$M$.
We build~$M$ successively by adding matching edges one after the other.
In the simplest case $M$ has no edges initially, but a set~$A$ of vertices of even cardinality marked as $\match$ in~$M$.
In other cases, some edges may already be prescribed or some vertices may be excluded, which is achieved by marking the corresponding vertices as~$\terminal$, $\uncovered$, or~$\forbidden$.
We use two approaches for generating all perfect matchings on the set~$A$.
The first one is based on depth-first search (DFS) and the second one on breadth-first search (BFS).
DFS takes~$M$, chooses a vertex $u \in A$, and for every other vertex $v \in A$ calls DFS for the matching obtained from~$M$ by adding the edge~$uv$ (by setting $M[u]:=v$ and $M[v]:=u$).
This approach stores only the current (partial) matching in memory, and tests a matching for extendability once it is completed to a perfect matching on~$A$.
Only the ones that are not extendable are stored for later processing.

The BFS method works as follows.

\begin{algorithm}[H]
\caption{BFS generation of perfect matchings on vertices in~$M$ marked as $\match$.}
Define $\cS:=\{M\}$\;
Let $a$ be the number of $\match$ vertices in $M$\;
\For{$a/2$ times}{
    Define $\cS':=\emptyset$\;
    \ForEach{$N \in \cS$}{
        Let $u$ be a vertex marked as $\match$ in $N$\;
        \ForEach{vertex $v\neq u$ marked as $\match$ in $N$}{
            Let $N'$ be the matching obtained from $N$ by adding the edge $uv$\;
            Set $\cS':=\cS'\cup \{N'\}$
        }
    }
    Set $\cS := \cS'$\;
    Remove isomorphic matchings from $\cS$
}
\end{algorithm}

BFS stores an entire set~$\cS$ of matchings simultaneously, which is more memory-intensive, but has the advantage that it allows removing isomorphic matchings each time an edge has been added (see the last line in the pseudocode above).
We therefore gain running time if we can quickly remove a significant number of isomorphic matchings.
Intuitively, the fraction of removed matchings decreases with the number of edges added, i.e., while many matchings with few edges will be removed as isomorphic, in later stages only few matchings with more edges will be removed.
For optimal performance, we apply BFS as long as there is enough memory available and the isomorphism testing removes a significant number of matchings.
Afterwards we continue building the matching via DFS.

We follow the approach from~\cite{MR3830138} to remove isomorphic matchings from a set~$\cS$ of matchings of~$Q_d$.
It is well-know that every automorphism of~$Q_d$ is composed of a transposition (i.e., switching certain directions) and a permutation of directions, i.e., there are $2^d d!$ automorphisms of~$Q_d$.
However, in our setting we require the fixed vertex $z=\emptyset$ to be forbidden, so we consider only automorphisms that permute directions, sometimes with one or two directions fixed.
Our function removes isomorphic matchings from~$\cS$ by applying to every matching $M\in \cS$ all relevant automorphisms and by selecting the lexicographically minimal representation.
As isomorphic matchings have the same lexicographically minimal representation, this will remove duplicates.

\subsubsection{Extending matchings to a cycle or linear forest}

We now explain how we test whether a matching~$M$ can be extended to a cycle or linear forest.
Specifically, if $M$ has no $\terminal$ vertices, we seek an extension to a cycle, and if $M$ has $t>0$ many $\terminal$ vertices, then we seek an extension to a linear forest consisting of $t/2$ paths between these $t$ terminals.
In both cases, the extension has to avoid all $\forbidden$ vertices, and it may use any number of $\uncovered$ vertices.
We test extendability using the following straightforward recursive function.

\begin{algorithm}[H]
\caption{Test whether a matching~$M$ can be extended to a cycle or linear forest.}
\SetKwFunction{Extends}{extends}
\myfunc{\Extends{$M$}}{
    Choose a vertex $u$ such that $M[u]$ is a vertex or $\terminal$\;
    \ForEach{$i \in [d]$}{
        \uIf{$uu^i$ {\rm is the last added edge}}{\KwRet{$\tttrue$}}
        \ElseIf{$uu^i$ {\rm can be added}}{
            Let $M'$ be obtained from $M$ by adding the edge $uu^i$\;
            \If{\Extends{$M'$}}{\KwRet{$\tttrue$}}
        }
    }
    \KwRet{$\ttfalse$}
}
\end{algorithm}

Adding an edge $uu^i$ is simple, although we have to consider a few cases depending on the values of~$M[u]$ and~$M[u^i]$.
If both $M[u]$ and $M[u^i]$ are vertices such that $M[u] \neq u^i$, we join the vertices~$M[u]$ and~$M[u^i]$ by an edge and we mark~$u$ and~$u^i$ as $\forbidden$.
Observe that the resulting matching~$M'$ can be extended if and only if $M \cup \{uu^i\}$ can be extended.
I.e., in our algorithm we store in~$M$ a matching that is a compact representation of the actual partial extension, obtained by replacing every path by the shortcut between its two end vertices.
Similarly, if $M[u]$ is a vertex and $M[u^i] = \uncovered$, we join vertices $M[u]$ and~$u^i$ by an edge and we mark~$u$ as~$\forbidden$.
If $M[u]$ is a vertex and $M[u^i]=\terminal$, we mark~$u$ and~$u^i$ as~$\forbidden$ and $M[u]$ as~$\terminal$.
The remaining cases are analogous.

It remains to determine whether an edge~$uu^i$ can be added and whether it is the last edge to be added.
Clearly, if $M[u]=u^i$ and $uu^i$ is not an edge of the original matching (but a shortcut edge), then adding the edge~$uu^i$ closes a cycle.
In order to avoid closing a cycle that misses some edges of~$M$, we maintain two counters for the number of edges and terminals in~$M$.
The last edge $uu^i$ is added if $M$ has one edge~$uu^i$ and no $\terminal$ vertices (extension to cycle), or $M$ has no edges and two $\terminal$ vertices~$u$ and~$u^i$ (extension to linear forest).
We can add an edge~$uu^i$ (not the last one) to the matching $M$ if $M[u] \neq u^i$ and $M$ has at least four terminals or at least one of~$u$ and $M[u]$ is not marked as~$\terminal$.

Our recursive algorithm works for every choice of~$u$, so we choose a vertex that minimizes the number of non-forbidden neighbors in~$Q_d$, which quickly cascades forced choices (if only one non-forbidden neighbor is left).
The number of non-forbidden neighbors for each vertex is maintained in an array.

\section{Proofs of Theorems~\ref{thm:long-Qd}, \ref{thm:long-KQd}, and~\ref{thm:Hamlace2-KQd}}
\label{sec:other-proofs}

\begin{proof}[Proof of Theorem~\ref{thm:long-Qd}]
Let $M$ be a matching of~$Q_d$, and let $M'$ be any maximal matching of~$Q_d$ that extends~$M$.
By Lemma~\ref{lem:max-Qd} we have $|M'|\ge \frac{d}{3d-1}|V(Q_d)|\ge \frac{1}{3}|V(Q_d)|$.
Applying Theorem~\ref{thm:cycle} to~$M'$ yields a cycle of length at least~$2|M'|=\frac{2}{3}|V(Q_d)|=2^{d+1}/3$ that extends~$M$.
\end{proof}

\begin{proof}[Proof of Theorem~\ref{thm:long-KQd}]
Let $M$ be a matching of~$K(Q_d)$, and let $M'$ be any maximal matching of~$K(Q_d)$ that extends~$M$.
By Lemma~\ref{lem:max-KQd} we have $|M'|\ge \frac{1}{4}|V(Q_d)|=2^{d-2}$.
Applying Theorem~\ref{thm:cycle} to~$M'$ yields a cycle of length at least~$2|M'|=\frac{1}{2}|V(Q_d)|=2^{d-1}$ that extends~$M$.
\end{proof}

\begin{proof}[Proof of Theorem~\ref{thm:Hamlace2-KQd}]
Let $M$ be a perfect matching of $K(Q_d\setminus\{x,y\})$ that extends to a cycle~$C$ that avoids~$x$ and~$y$, but contains all other vertices.
W.l.o.g.\ we assume that $x=\emptyset$.
Let $i\in[d]$.
If $y\in Q^i_0$, then as $y$ is avoided by~$M$ and has opposite (odd) parity, then $M$ cannot contain a half-layer in direction~$i$.
On the other hand, if~$y\in Q^i_1$, then $M$ cannot contain a half-layer in direction~$i$ either, as otherwise by property~\propH{} in Theorem~\ref{thm:forbidden} there would be a third vertex~$u\in V(Q^i_0) \setminus\{x\}$ avoided by~$M$, which is impossible.
We conclude that $M$ does not contain a half-layer.

To prove the converse direction, let $M$ be a perfect matching of~$K(Q_d\setminus\{x,y\})$ that does not contain a half-layer.
Then property~\propH{} in Theorem~\ref{thm:forbidden} is trivially satisfied, so the theorem yields a cycle~$C$ that extends~$M$ and avoids~$x$.
By Lemma~\ref{lem:avoid2}, $C$ also avoids~$y$, so the cycle~$C$ has the desired properties.
\end{proof}

\section{Open questions}
\label{sec:open}

We conclude this paper with some interesting directions for future investigations.

It would be interesting to translate the construction described in this paper to an algorithm that computes a cycle~$C$ that extends a given matching~$M$ of~$K(Q_d)$ in polynomial time (polynomial in the size of the graph~$Q_d$ or~$K(Q_d)$), and we do not see any fundamental obstacles to such an endeavor.
Furthermore, can this be improved, along the lines described in~\cite{MR4111723}, where part of the output can be computed already while knowing only part of the input?

Our Theorems~\ref{thm:long-Qd} and~\ref{thm:long-KQd} provide a lower bound for the length of a cycle~$C$ that extends a given matching~$M$.
Similarly, note that our proof of Theorem~\ref{thm:forbidden} starts by the assumption that $M$ is \propH{}-maximal, i.e., in a first step we always add edges to~$M$ so that we can guarantee the existence of a direction~$i\in[d]$ such that many edges of~$M$ have direction~$i$; recall Lemma~\ref{lem:H-maximal}.
This has the effect that even for very small matchings~$M$, the cycle~$C$ might be very long.
Can we instead prove a theorem about a cycle~$C$ extending a given matching~$M$, such that the length of~$C$ is relatively short compared to the size of~$M$, i.e., such that $|C|\leq f(|M|)$ for some reasonable function~$f$?
Obviously, we will always have $|C|\geq 2|M|$.

More generally, given a matching~$M$ of~$Q_d$ or~$K(Q_d)$, for which integers~$\ell$ is there a cycle~$C$ with $|C|=\ell$ that extends~$M$?
This includes the Ruskey-Savage conjecture as a special case (when $\ell=2^d$).
This can also be stated as a decision problem: Given $M$ and~$\ell$, does $C$ exist, yes or no?

Complementing Theorem~\ref{thm:vbwest}, it is known that not every matching of~$K(Q_d)$ can be extended to a cycle factor of~$Q_d$.
In fact, Dvo\v{r}\'{a}k and Fink~\cite{MR3891930} constructed a matching of~$B(Q_d)$ that cannot be extended to a cycle factor.
Is there a matching of~$K(Q_d)$ or~$B(Q_d)$ that is extendable to a cycle factor but not to a Hamilton cycle?

\bibliographystyle{alpha}
\bibliography{refs}
\end{document}